 \documentclass[11pt]{amsart}

\usepackage{amssymb}
\usepackage{amsthm}

\usepackage{amsmath}
\usepackage{mathrsfs}

\usepackage[all]{xy}%
\theoremstyle{plain}
    \newtheorem{theorem}{Theorem}[section]
    \newtheorem*{theorem*}{Theorem}
    \newtheorem{proposition}[theorem]{Proposition}
    \newtheorem{lemma}[theorem]{Lemma}
    \newtheorem{corollary}[theorem]{Corollary}
    \newtheorem*{corollary*}{Corollary}
\theoremstyle{definition}
    \newtheorem{definition}[theorem]{Definition}
    \newtheorem*{notation}{Notation}
    \newtheorem{example}[theorem]{Example}
    \newtheorem{remark}[theorem]{Remark}
\def\Alphabet{A,B,C,D,E,F,G,H,I,J,K,L,M,N,O,P,Q,R,S,T,U,V,W,X,Y,Z}
\def\alphabet{a,b,c,d,e,f,g,h,i,j,k,l,m,n,o,p,q,r,s,t,u,v,w,x,y,z}
\def\endpiece{xxx}
\def\makeAlphabet[#1]{\expandafter\makeA#1,xxx,}
\def\makealphabet[#1]{\expandafter\makea#1,xxx,}
\def\makeA#1,{\def\temp{#1}\ifx\temp\endpiece\else%
\mkbb{#1}\mkfrak{#1}\mkbf{#1}\mkcal{#1}\mkscr{#1}\expandafter\makeA\fi}%
\def\makea#1,{\def\temp{#1}\ifx\temp\endpiece\else\mkfrak{#1}\mkbf{#1}\expandafter\makea\fi}%
\def\mkbb#1{\expandafter\def\csname bb#1\endcsname{\mathbb{#1}}}
\def\mkfrak#1{\expandafter\def\csname fr#1\endcsname{\mathfrak{#1}}}
\def\mkbf#1{\expandafter\def\csname b#1\endcsname{\mathbf{#1}}}
\def\mkcal#1{\expandafter\def\csname c#1\endcsname{\mathcal{#1}}}
\def\mkscr#1{\expandafter\def\csname s#1\endcsname{\mathscr{#1}}}
\def\makeop[#1]{\xmakeop#1,xxx,}
\def\mkop#1{\expandafter\def\csname #1\endcsname{{\mathrm{#1}}\,}} %
\def\xmakeop#1,{\def\temp{#1}\ifx\temp\endpiece\else\mkop{#1}\expandafter\xmakeop\fi}%
\def\makesymb[#1]{\xmakesymb#1,xxx,}
\def\mksymb#1{\expandafter\def\csname #1\endcsname{{\mathrm{#1}}}} %
\def\xmakesymb#1,{\def\temp{#1}\ifx\temp\endpiece\else\mksymb{#1}\expandafter\xmakeop\fi}%
\makeAlphabet[\Alphabet]
\makealphabet[\alphabet]
\makeop[Spec, Spm, Spf]
\makesymb[Hom]
\makeop[Aut]
\makeop[Quot]
\makeop[Pic]
\makeop[Gal]
\makeop[coker, ker]
\makeop[mod]
\makeop[Idem]

\def\Fr{{\rm Fr}}
\def\Ker{{Ker}}
\def\dlim{\underrightarrow{\rm lim}}
\def\ilim{\underleftarrow{\rm lim}}

\def\Fr{\mathrm{Fr}}
\def\Ker{\mathrm{Ker}}
\def\Im{\mathrm{Im}}
\def\Id{\mathrm{Id}}

\def\Gal{\mathrm{Gal}}

\def\tr{\mathrm{tr}}

\def\Res{\mathrm{Res}}
\def\Cor{\mathrm{Cor}}

\def\dlim{\underrightarrow{\lim}}
\def\ilim{\underleftarrow{\lim}}
\makeop[codim]
\makeop[depth]
\makeop[Pic]
\makeop[mod]
\makeop[Idem]
\makeop[ord]
\def\pkappa{\widetilde \kappa}
\def\varep{\varepsilon}

\def\ur{\mathrm{ur}}
\def\tr{\mathrm{tr}}
\def\sgn{\mathrm{sgn}}
\def\Norm{\mathrm{N}}
\makeop[ind]

\begin{document}




\title[A generalization of Darmon's conjecture]{A generalization of Darmon's conjecture for Euler systems for general $p$-adic representations}
\author{Takamichi Sano}


\email{tkmc310@a2.keio.jp}
\address{Department of Mathematics, Keio University,\\3-14-1 Hiyoshi, Kohoku-ku, Yokohama, 223-8522, Japan}
\thanks{The author is supported by Grant-in-Aid for JSPS Fellows.}
\begin{abstract}
Darmon's conjecture on a relation between cyclotomic units over real quadratic fields and certain algebraic regulators was recently solved by Mazur and Rubin by using their theory of Kolyvagin systems. In this paper, we formulate a ``non-explicit" version of Darmon's conjecture for Euler systems defined for general $p$-adic representations, and 
prove it. In the process of the proof, we introduce a notion of ``algebraic Kolyvagin systems", and develop their properties.
\end{abstract}


%



\maketitle

\section{Introduction} 
One of the main themes in number theory is to study mysterious relations between values of zeta functions and arithmetical objects. Cyclotomic units are known as the elements which 
are closely related to the values of zeta functions. Their important properties are axiomatized in the theory of Euler systems, 
initiated by Kolyvagin (\cite{K}). The method of Euler systems is a powerful tool for bounding the size of Selmer groups. The collection of 
cyclotomic units is a typical example of an Euler system. 

In \cite{D}, inspired by the work of Gross (\cite{G}), Darmon conjectured a ``refined class number formula", which relates cyclotomic units over 
real quadratic fields with certain algebraic regulators. This conjecture is regarded as a refinement of classical class number formula of Dirichlet. 
Recently, Mazur and Rubin solved the ``except $2$-part" of Darmon's conjecture using their theory of Kolyvagin systems (\cite{MR2}). 

In this paper, as an attempt to generalize Darmon's conjecture, we formulate its analogue for Euler systems defined for general $p$-adic representations. Under the standard hypotheses 
of Kolyvagin systems (including that the core rank is equal to 1), we prove that our generalized Darmon's conjecture is true (see Theorem \ref{mainthm}). 
Our formulation is, however, not exactly a 
generalization of Darmon's conjecture. Let us explain the difference between our fomulation and the original conjecture of Darmon. Darmon's formulation is as follows. Let $F$ be a real quadratic field, and $\ell_1,\ldots, \ell_\nu$ be distinct prime numbers. Suppose that all $\ell_i$ 
split completely in $F$, for simplicity. Put $n=\ell_1 \cdots \ell_\nu$ ($n=1$ when $\nu=0$), and let $\zeta_n$ be 
a primitive $n$-th root of unity. Darmon defined the ``theta-element" $\theta_n'$ by 
$$\theta_n'=\sum_{\sigma \in {\rm Gal}(F(\zeta_n)/F)}\sigma\alpha_n\otimes \sigma \in F(\zeta_n)^\times \otimes_{\mathbb Z}\mathbb Z[{\rm Gal}(F(\zeta_n)/F)],$$
where $\alpha_n \in F(\zeta_n)^\times$ is a certain cyclotomic unit, and conjectured that the following equality holds in $(F(\zeta_n)^\times /\{ \pm 1\})\otimes_\mathbb Z \mathbb Z[{\rm Gal}(F(\zeta_n)/F)]/I_n^{\nu+1}$ ($I_n$ is the augmentation ideal of $\mathbb Z[{\rm Gal}(F(\zeta_n)/F)]$):
\begin{eqnarray}
\theta_n'=-h_nR_n, \label{dar}
\end{eqnarray}
where $h_n$ is the $n$-class number of $F$ (that is, the order of the Picard group of $\Spec(\cO_F[1/n])$), and $R_n \in F^\times \otimes_\mathbb Z I_n^\nu/I_n^{\nu+1}$ is an ``algebraic regulator". This algebraic regulator $R_n$ is defined by using a basis of $(1-\tau)\cO_F[1/n]^\times$ ($\tau$ is the non-trivial element of $\Gal(F/\bbQ)$) and the local reciprocity maps at places in $F$ lying above $\ell_i$'s. For more details, see \cite{D} and \cite{MR2}. 

Our formulation replaces the cyclotomic unit by an Euler system for a $p$-adic representation, the class number by the order of the Selmer group, the algebraic regulator 
by a {\it module} of algebraic regulators. 
Let $p$ be an odd prime, and $T$ be a $p$-adic representation of the absolute Galois group of $\mathbb{Q}$. 
Fix $M$, a power of $p$, and put $A=T/MT$. 
Let $\ell_1,\ldots,\ell_\nu$ be distinct prime numbers satisfying certain conditions. 
As before, put $n=\ell_1\cdots \ell_\nu$. Let $\mathbb Q(n)$ be the maximal $p$-subextension of $\mathbb Q$
inside $\mathbb Q(\zeta_n)$. Suppose that there exists an Euler system $c=\{c_n\}_n  \in \prod_n H^1(\mathbb Q(n),T)$ for $T$. Following Darmon, define the theta element $\theta_n(c)$ by 
$$\theta_n(c)=\sum_{\sigma \in {\rm Gal}(\mathbb Q(n)/\mathbb Q)}\sigma c_n \otimes \sigma.$$
We construct a submodule $\mathcal R_n \subset H^1(\mathbb Q,A)\otimes_\mathbb Z I_n^\nu/I_n^{\nu+1}$ as an analogue of Darmon's algebraic regulator, where $I_n$ is the augmentation ideal of $\bbZ[\Gal(\bbQ(n)/\bbQ)]$. Our formulation of the generalization of Darmon's conjecture is as follows: we have the following in $H^1(\mathbb Q(n),A)\otimes_\mathbb Z \mathbb Z[{\rm Gal}(\mathbb Q(n)/\mathbb Q)]/I_n^{\nu+1}$:
\begin{eqnarray}
\theta_n(c)\in |{\rm Sel}_{np}(\mathbb Q, A^\ast)|\mathcal R_n, \label{gdar}
\end{eqnarray}
where $A^\ast$ is the Kummer dual of $A$, ${\rm Sel}_{np}(\mathbb Q, A^\ast)$ is the Selmer group of $A^\ast$ of 
the elements which restrict to zero at the places dividing $np$, and $|\cdot|$ denotes the order.

Thus, our formulation (\ref{gdar}) is ``non-explicit" in the sense that the algebraic regulator in the right hand side is not determined explicitly. In this sense, our formulation is weaker than the original conjecture of Darmon  (\ref{dar}) when the Euler system is the system of cyclotomic units. But our formulation has a simple application. Taking $n=1$, we obtain the following:
$${\rm ord}_p(|{\rm Sel}_{p}(\mathbb Q, A^\ast)|) \leq {\rm Ind}(c),$$
where ${\rm Ind}(c)=\sup \{ m \ | \ c_1 \in p^m H^1(\mathbb Q,A) \} $. This is a quite famous result which is known as a typical application of the theory of Euler systems (see \cite[Theorem 2.2.2]{R} and \cite[Corollary 4.4.5]{MR1}). Our method gives another proof of this famous result.

A key observation of this paper lies in {\it defining a notion of ``algebraic Kolyvagin systems"}, which 
generalizes the notion of original Kolyvagin systems (see \S \ref{secKoly}). We define four different modules of algebraic Kolyvagin systems, called $\theta$-Kolyvagin systems, derived-Kolyvagin systems, pre-Kolyvagin systems, and (simply) Kolyvagin systems. The $\theta$-Kolyvagin system is the system whose axioms are satisfied by the collection $\{\theta_n(c)\}_n$ of the theta elements. The derived-Kolyvagin system is the system whose axioms are satisfied by the collection $\{\kappa_n'\}_n$ of the Kolyvagin's derivative classes of $c$. The pre-Kolyvagin system is an analogue of the $\theta$-Kolyvagin system. The system which we call simply Kolyvagin system is a direct generalization of the 
original Kolyvagin system. At a glance, these four modules of algebraic Kolyvagin systems may have different structures, but we prove that they are all isomorphic (see Theorem \ref{thmKS}). 
This observation is useful in some aspects; firstly, we can prove that $\{\theta_n(c)\}_n$ is a $\theta$-Kolyvagin system by reducing to show that the Kolyvagin's derivative classes $\{\kappa_n'\}_n$ of $c$ satisfy the axioms of the derived-Kolyvagin systems (see Proposition \ref{proptheta}); secondly, we can apply Mazur-Rubin's theory of Kolyvagin systems to other Kolyvagin systems. 

In \cite[Appendix B]{MR2}, Howard constructed ``regulator-type" Kolyvagin systems. We extend this construction to other Kolyvagin systems. We introduce a new system, which we call ``unit system", to treat Howard's construction more systematically (see Definition \ref{defu}). We interpret Howard's construction as a ``regulator map" from the module of unit systems to that of Kolyvagin systems (see Definition \ref{defreg}). We give analogues of this regulator map for other Kolyvagin systems, and prove the natural compatibility with the isomorphisms between different Kolyvagin systems (see Theorem \ref{thmreg}). We apply Mazur-Rubin's theory to know that the regulator map is surjective (see Theorem \ref{thmsurj}). From this, we know that the system of the theta elements, which forms a $\theta$-Kolyvagin system, is in the image of the regulator map. This says in fact that $\theta_n(c)\in |{\rm Sel}_{np}(\mathbb Q, A^\ast)|\mathcal R_n$ holds. Thus, we prove our main theorem. 

After the author wrote the first version of this paper, he found another generalization of Darmon's conjecture using Rubin-Stark elements (see \cite[Conjecture 3]{S}). He proved that, under some assumptions, most of this conjecture is deduced from the ``equivariant Tamagawa number conjecture (ETNC)" (\cite[Conjecture 4 (iv)]{BF}) for Tate motives (see \cite[Theorem 3.22]{S}). By the results of Burns and Greither (\cite{BG}), and of Flach (\cite{F}), we know that the ETNC for Tate motives is true for abelian fields. Using this fact, we know that the author's result (\cite[Theorem 3.22]{S}) gives another proof of the ``except $2$-part" of Darmon's conjecture, which was first proved by Mazur and Rubin in \cite{MR2} (see \cite[\S 4]{S} for the detail). The author hopes that there are some connections between the ETNC for general motives and the main result of this paper. We remark that, the paper \cite{MRGm} of Mazur and Rubin, in which essentially the same conjecture as \cite[Conjecture 3]{S} is formulated, appeared in arXiv after the author wrote the first version of the paper \cite{S} (see \cite[Conjecture 5.2]{MRGm}). 

The paper \cite{MRh} of Mazur and Rubin concerning higher rank Kolyvagin systems also appeared in arXiv after the author wrote this paper. In Definition \ref{defkoly}, the author proposed a definition of Kolyvagin systems of rank $r$, but when $r>1$ this definition is slightly different from the definition by Mazur and Rubin in \cite[Definition 10.4]{MRh}. The essentially same system as the ``unit system" introduced in this paper (see Definition \ref{defu}) is also defined by Mazur and Rubin in a rather neater way in \cite[Definition 7.1]{MRh}, and they call it ``Stark system". One of the main results of their paper (\cite[Theorem 12.4]{MRh}), which states that, under some assumptions, there is an isomorphism between ``Stark systems" and ``stub Kolyvagin systems", can be regarded as a generalization of our Theorem \ref{thmsurj}. 

The organization of this paper is as follows. \S \ref{secalg} is a short preparation from algebra. In \S \ref{statement}, after the setting of some notations and the review of some facts from Galois cohomology, we state our main theorem (Theorem \ref{mainthm}). In \S\S \ref{secKoly} and \ref{secreg}, we develop the theory of algebraic Kolyvagin systems. Finally in 
\S \ref{secpr} we prove Theorem \ref{mainthm}.

\begin{notation}

We fix an algebraic closure $\overline \bbQ$ of $\bbQ$, and any 
algebraic extension of $\bbQ$ is considered to be in $\overline \bbQ$. For each place $v$ of $\bbQ$, we choose a place $w$ of $\overline \bbQ$ 
above $v$, and fix it. By the decomposition (resp. inertia) group of $v$ in $G_\bbQ=\Gal(\overline \bbQ/\bbQ)$ we mean the decomposition (resp. inertia) group of $w$. The absolute Galois group of $\bbQ_v$ is identified with the decomposition group of $v$ in $G_\bbQ$. 

For $m \in \bbZ_{\geq1}$, $\mu_m$ denotes the group of all the $m$-th roots of unity in $\overline \bbQ$.

For a field $F$, and a continuous $\Gal(\overline F/F)$-module $M$ (where $\overline F$ is a fixed separable closure of $F$), we denote
$$H^i(F,M)=H^i_{\rm{cont}}(\Gal(\overline F/F),M),$$
where $H^i_{\rm{cont}}$ is the continuous cochain cohomology (\cite{T}). 

If $G$ is a profinite group, and $M$ is a continuous $G$-module, we denote for $\tau \in G$ 
$$M^{\tau=1}=\{a \in M \ | \ \tau a=a  \}.$$
\end{notation}

\section{Exterior powers} \label{secalg}

Let $A$ be a ring, $B$ be an $A$-algebra, and $M$ be an $A$-module. If $f \in \Hom_A(M, B)$, then there is an $A$-homomorphism
$$\bigwedge_A^rM \longrightarrow \left( \bigwedge_A^{r-1}M \right) \otimes_A B $$
$$ m_1 \wedge \cdots  \wedge m_r \mapsto \sum_{i=1}^{r}(-1)^{i-1}m_1 \wedge \cdots \wedge m_{i-1}
\wedge m_{i+1} \wedge \cdots \wedge m_r\otimes f(m_i), $$
for all $r \in \bbZ_{\geq 1}$. This map is also denoted by $f$. Moreover, the induced $B$-homomorphism 
$$\left(\bigwedge_A^rM \right) \otimes_A B \longrightarrow \left( \bigwedge_A^{r-1}M \right) \otimes_A B $$
is also denoted by $f$. Then the following $B$-homomorphism is well-defined for $r, s \in \bbZ_{\geq 1}$ with $r \geq s$ :
$$ \bigwedge_B^s \Hom_A(M, B) \longrightarrow \Hom_A\left(  \bigwedge_A^rM , \left( \bigwedge_A^{r-s}M \right)\otimes_A B \right), $$
$$ f_1\wedge \cdots \wedge f_s \mapsto \left(m_1\wedge \cdots \wedge m_r \mapsto (f_1\circ \cdots \circ f_s)(m_1\wedge \cdots \wedge m_r)\right).$$
We define $(f_1 \wedge \cdots \wedge f_s)(m_1\wedge \cdots \wedge m_r)$ by
$$(f_1\circ \cdots \circ f_s)(m_1\wedge \cdots \wedge m_r).$$
Note that our construction of the map $f_1\wedge \cdots \wedge f_s$ is slightly different from that in \cite[\S 1.2]{R2} ($f_1\wedge \cdots \wedge f_s$ is defined by $f_s \circ \cdots \circ f_1$ in loc. cit.).

\begin{remark}
Using the above notations, we have
$$(f_1\wedge \cdots \wedge f_r)(m_1\wedge \cdots \wedge m_r)=\left|
\begin{array}{ccc}
	f_r(m_1) &\cdots &f_r(m_r) \\ 
	\vdots & \ddots & \vdots \\
	f_1(m_1) & \cdots & f_1(m_r)
\end{array}
\right| \in B$$
and 
$$(f_1\wedge \cdots \wedge f_{r-1})(m_1\wedge \cdots \wedge m_r)=\left|
\begin{array}{cccc}
	f_{r-1}(m_1) &\cdots & &f_{r-1}(m_r) \\ 
	\vdots & \ddots & & \vdots \\
	f_1(m_1) & \cdots & & f_1(m_r)\\
	m_1 & \cdots & & m_r
\end{array}
\right| \in M\otimes_A B.$$
The first equality is shown by induction on $r$ as follows. When $r=1$, the left hand side is equal to $f_1(m_1)$ by definition, and this is equal to the 
right hand side. When $r>1$, note first that by the inductive hypothesis we see that $(f_1 \wedge \cdots \wedge f_{r-1})(m_1 \wedge \cdots \wedge m_{i-1} \wedge m_{i+1} \wedge \cdots \wedge m_r)$ is equal to the determinant of the matrix obtained by removing the first row and the $i$-th column from the matrix of the right hand side. We have
\begin{eqnarray}
& &(f_1\wedge \cdots \wedge f_r)(m_1\wedge \cdots \wedge m_r) \nonumber \\
&=&\sum_{i=1}^r(-1)^{i-1}f_r(m_i)(f_1 \wedge \cdots \wedge f_{r-1})(m_1 \wedge \cdots \wedge m_{i-1} \wedge m_{i+1} \wedge \cdots \wedge m_r) \nonumber \\
&=&\left|
\begin{array}{ccc}
	f_r(m_1) &\cdots &f_r(m_r) \\ 
	\vdots & \ddots & \vdots \\
	f_1(m_1) & \cdots & f_1(m_r)
\end{array}
\right|, \nonumber 
\end{eqnarray}
where the second equality is obtained by the cofactor expansion along the first row. This completes the proof.
\end{remark}

\section{The statement of the main theorem} \label{statement}

The aim of this section is to state our main theorem (Theorem \ref{mainthm}). First, we set some notations. Let $p$ be an odd prime, and fix a power of $p$, which is denoted by $M$. Let $T$ be a $p$-adic representation 
of the absolute Galois group of $\bbQ$ with coefficients in $\bbZ_p$, that is, $T$ is a free $\bbZ_p$-module of finite rank with a continuous $\bbZ_p$-linear action of $G_\bbQ=\Gal(\overline \bbQ /\bbQ)$. As usual, we assume that $T$ is unramified at almost all places of $\bbQ$, that is, for all but finitely many places $v$ of $\bbQ$, the inertia group of $v$ in $G_\bbQ$ acts trivially on $T$. We write $A=T/MT$. Fix $\Sigma$, a set of places of $\bbQ$, such that
$$\Sigma \subset \{ \ell \ | \ \mbox{$\ell$ is a prime satisfying $(*)$} \},$$
where $(*)$ is as follows: 
$$(*) \left\{
\begin{array}{l}
T \mbox{ is unramified at } \ell, \\
M \mbox{ divides } \ell -1, \\
A/(\Fr_\ell-1)A \simeq \bbZ/M\bbZ, \\
\end{array} \right.$$
where $\Fr_\ell $ is the arithmetic Frobenius at $\ell$. 

Next, put $\cN=\cN(\Sigma)=\{ \mbox{ square-free products of primes in }\Sigma  \ \}$. We suppose $1\in\cN$, for convention. Note that 
$\cN$ is naturally identified with the family of all the finite subsets of $\Sigma$ (with this identification, $1 \in \cN$ corresponds to the empty set $\emptyset \subset \Sigma$). This observation will be used later, in \S \ref{secKoly}. 

For every $\ell \in \Sigma$, put 
$$P_\ell(x)=\det(1-\Fr_\ell x|T)\in \bbZ_p[x],$$
where the right hand side means the characteristic polynomial with respect to the action of $\Fr_\ell$ on $T$. Note that $P_\ell(1) \equiv 0 \ (\mod \ M)$, since $A/(\Fr_\ell-1)A \simeq \bbZ/M\bbZ$ (see \cite[Lemma 1.2.3]{MR1}). Put
$$Q_\ell(x)=\frac{P_\ell(x)-P_\ell(1)}{x-1} \ \mod \ M \in \bbZ/M\bbZ[x].$$
This is the unique polynomial such that 
$$(x-1)Q_\ell(x)\equiv P_\ell(x) \ \mod \ M$$
(see \cite[Lemma 4.5.2]{R} or \cite[Definition 1.2.2]{MR1}).

Next, for every $n \in \cN$, put
$$G_n=\Gal(\bbQ(n)/\bbQ),$$
where $\bbQ(n)$ is the maximal $p$-subextension of $\bbQ$ inside $\bbQ(\mu_n)$. Note that we have a natural isomorphism 
$G_n\simeq\bigoplus_{\ell | n}G_\ell.$ For every $\ell \in \Sigma$, we define a generator $\sigma_\ell$ of $G_\ell$ as follows. Fix 
a generator $\xi$ of $\bbZ_p$-module $\ilim \mu_{p^m}$. Since we fixed the embedding $\overline \bbQ \hookrightarrow \overline \bbQ_\ell$, $\ilim \mu_{p^m}$ is also regarded as a subgroup of $\ilim \overline \bbQ_\ell^\times$. By Kummer theory, we have a canonical isomorphism 
$$\Gal(\bbQ_\ell^{\ur}(\ell^{1/{p^\infty}})/\bbQ_\ell^{\ur}) \stackrel{\sim}{\longrightarrow} \ilim \mu_{p^m} \quad ; \quad \sigma \mapsto (\frac{\sigma (\ell^{1/{p^m}})}{\ell^{1/{p^m}}})_m,$$
 where $\bbQ_\ell^{\ur}$ is the maximal unramified extension of $\bbQ_\ell$. We also have a natural surjection  $\Gal(\bbQ_\ell^{\ur}(\ell^{1/{p^\infty}})/\bbQ_\ell^{\ur}) \rightarrow G_\ell$, so we have a surjection $\ilim \mu_{p^m} \rightarrow G_\ell$. We define 
$\sigma_\ell \in G_\ell$ to be the image of $\xi\in \ilim \mu_{p^m}$ by this surjection.

For $n \in \cN$, we denote $I_n$ the augmentation ideal of $\bbZ[G_n]$. Note that if ${\ell} {\not{|}} n$, then we have 
$$P_\ell(\Fr_\ell) \otimes 1 \in I_n \otimes \bbZ/M\bbZ,$$
since $P_\ell(1) \equiv 0 \ (\mod M)$ as we mentioned above, where $\Fr_\ell$ is naturally regarded as an element of $G_n$ (note that since $\ell$ is prime to $n$, $\ell$ is unramified in $\bbQ(n)$). Therefore, we consider the image of $P_\ell(\Fr_\ell) \otimes 1$ in $I_n/I_n^2 \otimes \bbZ/M\bbZ$, and denote it also by $P_\ell(\Fr_\ell)\otimes 1$.

We next define important maps $v_\ell$, $u_\ell$, and $\varphi_\ell$ for $\ell \in \Sigma$. As a preliminary, we review some facts on Galois cohomology.

For $\ell \in \Sigma$, the unramified cohomology group at $\ell$ is defined by
$$H_{\ur}^1(\bbQ_\ell,A)=H^1(\bbQ_\ell ^{\ur}/\bbQ_\ell,A).$$
There is a canonical isomorphism:
$$H_{\ur}^1(\bbQ_\ell,A) \simeq A/(\Fr_\ell-1)A,$$
which is obtained by evaluating $\Fr_\ell\in\Gal(\bbQ_\ell ^{\ur}/\bbQ_\ell)$ to 1-cocycles representing elements of $H_{\ur}^1(\bbQ_\ell,A)$  (see \cite[Lemma B.2.8]{R} or \cite[Lemma 1.2.1 (i)]{MR1}).

There is a canonical decomposition:
$$H^1(\bbQ_\ell,A)\simeq H_{\tr}^1(\bbQ_\ell,A)\oplus H_{\ur}^1(\bbQ_\ell,A),$$
where $H_{\tr}^1(\bbQ_\ell,A):=H^1(\bbQ_\ell(\mu_\ell)/\bbQ_\ell,A^{G_{\bbQ_\ell(\mu_\ell)}})$ is called the transverse cohomology group at $\ell$, and naturally identified with $\Hom(G_\ell,A^{\Fr_\ell=1})$ (see \cite[Lemma 1.2.1 (ii) and Lemma 1.2.4]{MR1}). 
We remark that to get this decomposition, the assumption $M| \ell-1$ is needed. 

Now we start to define $v_\ell$, $u_\ell$, and $\varphi_\ell$.

First, the definition of $v_\ell$ is as follows: 
\begin{eqnarray}
v_\ell :  H^1(\bbQ, A) &\longrightarrow& H^1(\bbQ_\ell, A) \nonumber \\
&\longrightarrow& H^1_{\tr}(\bbQ_\ell,A)=\Hom(G_\ell,A^{\Fr_\ell=1}) \nonumber \\
&\stackrel{\sim}{\longrightarrow}& A^{\Fr_\ell=1} \simeq \bbZ/M\bbZ, \nonumber
\end{eqnarray}
where the first arrow is the localization map at $\ell$, the second is the natural projection, the third isomorphism is 
obtained by evaluating $\sigma_\ell \in G_\ell$ (recall that $\sigma_\ell$ is a fixed generator of $G_\ell$), and the last (non-canonical) isomorphism follows by noting that $A/(\Fr_\ell-1)A\simeq\bbZ/M\bbZ$ (see \cite[Lemma 1.2.3]{MR1}). We fix the last isomorphism.

Next, we define the map $u_\ell$ as follows:
\begin{eqnarray}
u_\ell  :  H^1(\bbQ,A) &\longrightarrow& H^1(\bbQ_\ell,A) \nonumber \\
&\longrightarrow& H^1_{\ur}(\bbQ_\ell,A)=A/(\Fr_\ell-1)A \nonumber \\
&\stackrel{-Q_\ell(\Fr_\ell^{-1})}{\longrightarrow}& A^{\Fr_\ell=1}=\bbZ/M \bbZ, \nonumber
\end{eqnarray}
where the first arrow is the localization at $\ell$, and the second is the natural projection. The third arrow is defined by 
$$A/(\Fr_\ell-1)A \longrightarrow A^{\Fr_\ell=1}\quad ; \quad \bar a \mapsto -Q_\ell(\Fr_\ell^{-1})a $$
(the well-definedness is easily verified by using the Cayley-Hamilton theorem). This is in fact an isomorphism, see \cite[Corollary A.2.7]{R} for the proof. Note that we use $-Q_\ell(\Fr_\ell^{-1})$ instead of $Q_\ell(\Fr_\ell^{-1})$ (this turns out to be meaningful when we see Example \ref{exrec} below). The last identification $A^{\Fr_\ell=1}=\bbZ/M\bbZ$ in the definition of $u_\ell$ above is obtained by the fixed isomorphism when we defined $v_\ell$.

Finally, we define $\varphi_\ell$ as follows:
$$\varphi_\ell:H^1(\bbQ,A) \longrightarrow \mathop{\ilim}_{n \in\cN}(I_{n}/I_{n}^2 \otimes \bbZ/M\bbZ) \quad ; \quad a \mapsto -(\sigma_\ell-1) \otimes u_\ell(a) -P_\ell(\Fr_\ell) \otimes v_\ell(a),$$
where the inverse limit in the right hand side is taken with respect to the natural restriction map of Galois groups, namely, if $n, m \in\cN$ and $n |m$, the morphism from $I_m/I_m^2 \otimes \bbZ/M\bbZ$ to $I_n/I_n^2 \otimes \bbZ/M\bbZ$ is induced by the natural surjection $G_m \rightarrow G_n$. Note that $P_\ell(\Fr_\ell) \otimes 1$ is naturally 
regarded as an element of $\mathop{\ilim}_{n \in\cN, \ell|n}(I_{n/\ell}/I_{n/\ell}^2 \otimes \bbZ/M\bbZ)$. Since we have the canonical isomorphism 
$$(I_\ell/I_\ell^2 \otimes \bbZ/M\bbZ)\oplus \mathop{\ilim}_{n \in\cN, \ell | n} (I_{n/\ell}/I_{n/\ell}^2 \otimes \bbZ/M\bbZ) \simeq \mathop{\ilim}_{n \in\cN}(I_n/I_n^2 \otimes \bbZ/M\bbZ),$$
we see that $-(\sigma_\ell-1)\otimes u_\ell(a)-P_\ell(\Fr_\ell) \otimes v_\ell(a)$ lies in $\mathop{\ilim}_{n \in\cN}(I_n/I_n^2 \otimes \bbZ/M\bbZ)$, hence $\varphi_\ell$ is defined. 

\begin{example} \label{exrec}
Take $T=\bbZ_p(1)=\ilim \mu_{p^m}$, and $A=T/MT=\mu_M$. Take $\ell \in \Sigma$. Suppose $a\in \bbQ^\times/(\bbQ^\times)^M \simeq H^1(\bbQ,A)$, and 
$$a=\ell^i e \quad \mbox{in }\bbQ_\ell^\times/(\bbQ_\ell^\times)^M,$$
where $i\in \bbZ/M\bbZ$ and $e\in \mu_M$ (note that $i$ and $e$ are uniquely determined for the image of $a$ in $\bbQ_\ell^\times/(\bbQ_\ell^\times)^M$). 
If we identify $\bbZ/M\bbZ=\mu_M$ by fixing a primitive $M$-th root of unity, then we see that 
$$v_\ell(a)=i$$
and 
$$u_\ell(a)=e$$
(note that since $P_\ell(x)=1-\ell x \equiv 1-x \ (\mod \ M)$, we have $Q_\ell(x)=-1$). We see that $\varphi_\ell$ agrees with the following map:
\begin{eqnarray}
H^1(\bbQ,A) \simeq \bbQ^\times/(\bbQ^\times)^M &\longrightarrow& \bbQ_\ell^\times/(\bbQ_\ell^\times)^M \nonumber \\
& \stackrel{{\rm{rec}}_\ell}{\longrightarrow}& \ilim (G_n \otimes \bbZ/M\bbZ) \nonumber \\
&\stackrel{\sim}{\longrightarrow}& \ilim (I_n/I_n^2 \otimes \bbZ/M\bbZ),\nonumber 
\end{eqnarray}
where ${\rm{rec}}_\ell$ is the map induced by the local reciprocity map at $\ell$, and the last isomorphism is given by $\sigma \mapsto \sigma-1$.
\end{example}

We put 
$$G(n)=\bigoplus_{i=0}^\infty I_n^i/I_n^{i+1}\otimes \bbZ/M\bbZ$$ 
for $n\in\cN$, where $I_n^0$ is understood to be $\bbZ[G_n]$ (so we have $I_n^0/I_n^1=\bbZ$). $G(n)$ has a structure of graded $\bbZ/M\bbZ$-algebra, and we can regard $\varphi_\ell$ as a homomorphism from $H^1(\bbQ,A)$ to a $\bbZ/M\bbZ$-module $\mathop{\ilim}_{n \in \cN}G(n)$, that is, $\varphi_\ell \in \Hom_{\bbZ/M\bbZ}(H^1(\bbQ,A), \mathop{\ilim}_{n \in \cN}G(n))$.

We define $\varphi_\ell^n$ to be the composition of the projection to $G(n)$ followed by $\varphi_\ell$, that is,
$$\varphi_\ell^n: H^1(\bbQ,A) \stackrel{\varphi_\ell}{\longrightarrow} \mathop{\ilim}_{n \in \cN}G(n) \longrightarrow G(n).$$

We denote throughout this paper $\cF$ the canonical Selmer structure on $T$ in the sense of \cite[Definition 3.2.1]{MR1}. For $n \in\cN$, we recall that the $n$-modified Selmer group $H^1_{\cF^n}(\bbQ,A)$ is defined by
$$H^1_{\cF^n}(\bbQ,A)=\{ a\in H^1(\bbQ,A) \ | \ a_\ell \in H^1_{\cF}(\bbQ_\ell,A) \mbox{ for any }\ell {\not{|}} n \},$$
where $a_\ell$ is the image of $a$ by the localization at $\ell$. We also recall that the $n$-strict dual Selmer group $H^1_{(\cF^\ast)_n}(\bbQ,A^\ast)$ is defined by
$$H^1_{(\cF^\ast)_n}(\bbQ,A^\ast)=\{ a \in H^1_{\cF^\ast}(\bbQ,A^\ast) \ | \ a_\ell=0 \mbox{ for any }\ell|n \},$$
where $A^\ast=\Hom(A,\mu_M)$ is the Kummer dual of $A$, and $\cF^\ast$ is the dual Selmer structure of $\cF$. See \cite[Example 2.1.8 and Definition 2.3.1]{MR1}.
\begin{definition} \label{defreg1}
For $n\in\cN$, we define a (module of) regulator $\cR_n$ by
$$\cR_n=\Im\left(\varphi_{\ell_1}^n\wedge\cdots\wedge\varphi_{\ell_{\nu(n)}}^n: \bigwedge_{\bbZ_p}^{\nu(n)+1}H^1_{\cF^n}(\bbQ,A)\longrightarrow H^1_{\cF^n}(\bbQ,A)\otimes G(n)\right),$$
where $n=\ell_1\cdots\ell_{\nu(n)} $ and $\nu(n)$ is the number of prime divisors of $n$ (for the definition of the map  $\varphi_{\ell_1}^n\wedge\cdots\wedge\varphi_{\ell_{\nu(n)}}^n$, see \S \ref{secalg}). Note that $\cR_n$ does not depend on the 
choice of the order of $\ell_1,\ldots,\ell_{\nu(n)}$.
\end{definition}

We recall the definition of Euler systems (\cite[Definition 2.1.1]{R}, \cite[Definition 3.2.2]{MR1}). Note that the definition of Euler systems in \cite{R} and that in \cite{MR1} are slightly different (see \cite[Remark 3.2.3]{MR1}). Our definition is due to the latter.

\begin{definition}
A collection 
$$\{ c_F\in H^1(F,T) \ | \ \bbQ \subset F \subset \cK, \ F/\bbQ \mbox{ : finite extension} \}$$
is an Euler system for $(T, \Sigma, \cK)$, where $\cK$ is an abelian extension of $\bbQ$, if, whenever $F \subset F' \subset \cK$ and $F'/\bbQ$ is finite,
$$\Cor_{F'/F}(c_{F'})=\left( \prod P_\ell(\Fr_\ell^{-1}) \right)c_F,$$
where the product runs over primes $\ell \in \Sigma$ which ramify in $F'$ but not in $F$.
\end{definition}

We define an analogue of Darmon's ``theta-element" (\cite[\S 4]{D}) for a general Euler system.

\begin{definition} \label{deftheta}  
Suppose $c=\{  c_F\in H^1(F,T) \ | \ \bbQ \subset F \subset \cK, \ F/\bbQ \mbox{ : finite extension} \}$ is an Euler system for $(T, \Sigma, \cK)$ such that $\bbQ(n) \subset \cK$ for any $n\in\cN$. We define the theta element $\theta_n(c)$ for $n\in\cN$ by 
$$\theta_n(c)=\sum_{\gamma \in G_n}\gamma c_n \otimes \gamma \in H^1(\bbQ(n), A)\otimes \bbZ[G_n],$$
where $c_n=c_{\bbQ(n)}$ (which we regard as an element of $H^1(\bbQ(n),A)$ via the natural map $H^1(\bbQ(n),T)\rightarrow H^1(\bbQ(n),A)$, induced by the natural surjection $T\rightarrow A$).
\end{definition}

\begin{lemma} \label{lemth}   
Suppose $d,n\in\cN$ and $d|n$. Then we have 
$$\pi_d(\theta_n(c))=\theta_d(c)\prod_{\ell | n/d}P_\ell(\Fr_\ell),$$
where $\pi_d$ is the map induced by the natural projection $G_n\rightarrow G_d$.
\end{lemma}

\begin{proof}
We may assume $d=n/\ell$, where $\ell$ is a prime divisor of $n$. We compute 
\begin{eqnarray}
\pi_{n/\ell}(\theta_n(c)) &=& \pi_{n/\ell}\left( \sum_{\gamma \in G_n}\gamma c_n \otimes \gamma \right) \nonumber \\
&=&\sum_{\alpha \in G_{n/\ell}}\sum_{\beta \in G_\ell} \alpha \beta c_n \otimes \alpha \nonumber \\
&=&\sum_{\alpha \in G_{n/\ell}}\alpha\cdot \Norm_{\bbQ(n)/\bbQ(n/\ell)}(c_n)\otimes \alpha \nonumber \\
&=& \sum_{\alpha \in G_{n/\ell}}\alpha\cdot P_\ell(\Fr_\ell^{-1})c_{n/\ell} \otimes \alpha \nonumber \\
&=& \sum_{\alpha \in G_{n/\ell}}\alpha c_{n/\ell}\otimes \alpha\cdot P_\ell(\Fr_\ell) \nonumber \\
&=& \theta_{n/\ell}(c)P_\ell(\Fr_\ell), \nonumber
\end{eqnarray}
where $\Norm_{\bbQ(n)/\bbQ(n/\ell)}$ is the norm from $\bbQ(n)$ to $\bbQ(n/\ell)$ (note that $\Norm_{\bbQ(n)/\bbQ(n/\ell)}$ is equal to $\Res_{\bbQ(n)/\bbQ(n/\ell)}\circ \Cor_{\bbQ(n)/\bbQ(n/\ell)}$). This proves the lemma.
\end{proof}

The following proposition is an analogue of \cite[Theorem 4.5 (2)]{D}. 

\begin{proposition} \label{proptheta1}   
Let the notations be as in Definition \ref{deftheta}. We have 
$$\theta_n(c) \in H^1(\bbQ(n),A) \otimes I_n^{\nu(n)},$$
and if we regard $\theta_n(c)\in H^1(\bbQ(n),A) \otimes I_n^{\nu(n)}/I_n^{\nu(n)+1}$, then there is a canonical inverse image of $\theta_n(c)$ 
under the restriction map 
$$H^1(\bbQ,A)\otimes I_n^{\nu(n)}/I_n^{\nu(n)+1} \longrightarrow H^1(\bbQ(n),A)\otimes I_n^{\nu(n)}/I_n^{\nu(n)+1},$$
namely, there is a canonical element $x_n\in H^1(\bbQ,A)\otimes I_n^{\nu(n)}/I_n^{\nu(n)+1}$ such that $\Res_{\bbQ(n)/\bbQ}(x_n)=\theta_n(c)$.
\end{proposition}

\begin{proof}
We prove this proposition by induction on $\nu(n)$. When $\nu(n)=0$ (i.e. $n=1$), we have $I_1^0=\bbZ$ and $\theta_1(c)=c_1\in H^1(\bbQ,A)$, so there is nothing to prove (since $x_1=c_1$). Suppose $\nu(n)>0$. We write every $\gamma \in G_n$ uniquely as 
$$\gamma=\prod_{\ell|n}\gamma_\ell,$$
where $\gamma_\ell \in G_\ell$. We compute
\begin{eqnarray}
\sum_{\gamma \in G_n}\gamma c_n\otimes\prod_{\ell|n}(\gamma_\ell-1) &=&\theta_n(c)+\sum_{d|n,d\neq n}(-1)^{\nu(n/d)}\left(\sum_{\gamma\in G_n}\gamma c_n\otimes\prod_{\ell|d}\gamma_\ell\right) \nonumber \\
&=&\theta_n(c)+\sum_{d|n,d\neq n}(-1)^{\nu(n/d)}\theta_d(c)\prod_{\ell|n/d}P_\ell(\Fr_\ell), \nonumber 
\end{eqnarray}
where the first equality follows by direct computation, and the second by Lemma \ref{lemth}. This shows 
$\theta_n(c) \in H^1(\bbQ(n), A)\otimes I_n^{\nu(n)}$, since by the inductive hypothesis we have $$\theta_d(c)\prod_{\ell|n/d}P_\ell(\Fr_\ell)\in H^1(\bbQ(n),A)\otimes I_n^{\nu(n)}$$ if $d|n$ and $d\neq n$.

We compute
$$\sum_{\gamma \in G_n}\gamma c_n\otimes\prod_{\ell|n}(\gamma_\ell-1)=\left( \prod_{\ell|n}D_\ell \right)c_n\otimes \prod_{\ell|n}(\sigma_\ell-1)  \mbox{ in } H^1(\bbQ(n),A)\otimes I_n^{\nu(n)}/I_n^{\nu(n)+1}, $$
where 
$$D_\ell=\sum_{i=1}^{|G_\ell|-1}i\sigma_\ell^i,$$
(recall that $\sigma_\ell$ is a fixed generator of $G_\ell$). It is well known that $\left( \prod_{\ell|n}D_\ell \right)c_n $ has a 
canonical inverse image in $H^1(\bbQ,A)$, which is usually called Kolyvagin's derivative class (see \cite[Definition 4.4.10]{R}). We denote it by $\kappa'_n$ (in \cite[\S 4.4]{R}, it is denoted by $\kappa_{[\bbQ,n,M]}$). 
Hence we have 
\begin{equation}
\theta_n(c)=\kappa'_n\otimes\prod_{\ell|n}(\sigma_\ell-1)-\sum_{d|n,d\neq n}(-1)^{\nu(n/d)}\theta_d(c)\prod_{\ell|n/d}P_\ell(\Fr_\ell). \label{eq10}
\end{equation}
By the inductive hypothesis, we see that $\theta_n(c)\in H^1(\bbQ(n),A) \otimes I_n^{\nu(n)}/I_n^{\nu(n)+1}$ has a canonical 
inverse image in $H^1(\bbQ,A) \otimes I_n^{\nu(n)}/I_n^{\nu(n)+1}$.
\end{proof}

\begin{remark}
By the proof of Proposition \ref{proptheta1}, the element $x_n \in H^1(\bbQ,A)\otimes I_n^{\nu(n)}/I_n^{\nu(n)+1}$ such that $\Res_{\bbQ(n)/\bbQ}(x_n)=\theta_n(c)$ 
is inductively constructed by 
$$x_n=\kappa'_n\otimes\prod_{\ell|n}(\sigma_\ell-1)-\sum_{d|n,d\neq n}(-1)^{\nu(n/d)}x_d\prod_{\ell|n/d}P_\ell(\Fr_\ell).$$
Since $\kappa_n'$ is a canonical element, we can say that $x_n$ is also canonical. So we can naturally regard $\theta_n(c)\in H^1(\bbQ,A)\otimes I_n^{\nu(n)}/I_n^{\nu(n)+1}$.
\end{remark}

We summarize here the standard hypotheses (H.0)-(H.6) of Kolyvagin systems for the triple $(A, \cF, \Sigma)$ (\cite[\S 3.5]{MR1}):
\begin{itemize}
\item[(H.0)]{$A$ is a free $\bbZ/M\bbZ$-module of finite rank.}
\item[(H.1)]{$A/pA$ is an absolutely irreducible $\bbF_p[G_\bbQ]$-representation.}
\item[(H.2)]{There is a $\tau\in G_\bbQ$ such that $\tau=1$ on $\mu_{p^\infty}$ and $A/(\tau-1)A \simeq \bbZ/M\bbZ$.}
\item[(H.3)]{$H^1(\bbQ(A)\bbQ(\mu_{p^\infty}),A/pA)=H^1(\bbQ(A)\bbQ(\mu_{p^\infty}),A^\ast[p])=0$, where $\bbQ(A)$ is the fixed field in $\overline \bbQ$ of the kernel of the map $G_\bbQ \rightarrow \Aut(A)$, and $A^\ast[p]=\{a \in A^\ast \ | \ pa=0\}$.}
\item[(H.4)]{Either
\begin{itemize}
\item[(H.4a)]{$\Hom_{\bbF_p[[G_\bbQ]]}(A/pA,A^\ast[p])=0$, or}
\item[(H.4b)]{$p>4$.}
\end{itemize}}
\item[(H.5)]{$\Sigma_t \subset \Sigma \subset \Sigma_1$ for some $t\in \bbZ_{>0}$, where for $k\in\bbZ_{>0}$ $\Sigma_k$ is the set of all the primes $\ell$ satisfying ($*$) for $M$ replaced by $p^k$.} 
\item[(H.6)]{For every $\ell \in \{ \ell \ | \ T \mbox{ is ramified at } \ell \} \cup \{ p, \infty\}$, the local condition $\cF$ at $\ell$ is cartesian (see \cite[Definition 1.1.4]{MR1}) on the category $\Quot_{\bbZ/M\bbZ}(A)$ (see \cite[Example 1.1.3]{MR1}).}
\end{itemize}

Note that, in our case, (H.0) is always satisfied.


Now, our main theorem is as follows:

\begin{theorem} \label{mainthm}
Suppose that there exists an Euler system $c$ for $(T,\Sigma,\cK)$. Assume the following:
\begin{itemize}
\item[(i)]{the standard hypotheses {\rm{(H.0)-(H.6)}} of Kolyvagin systems are satisfied for the triple $(A, \cF, \Sigma)$,}
\item[(ii)]{$\cK$ contains the maximal abelian $p$-extension of $\bbQ$ which is unramified outside of $p$ and $\Sigma$,}
\item[(iii)]{$T/(\Fr_\ell-1)T$ is a cyclic $\bbZ_p$-module for every $\ell \in\Sigma$,}
\item[(iv)]{$\Fr_\ell^{p^k}-1$ is injective on $T$ for every $\ell\in\Sigma$ and $k\geq0$,}
\item[(v)]{the core rank $\chi(A, \cF)=1$ (\cite[Definition 4.1.11]{MR1}),} 
\end{itemize}
((ii)-(iv) are the assumptions of the first statement of \cite[Theorem 3.2.4]{MR1}, and (iii) is satisfied since we assumed $A/(\Fr_\ell-1)A \simeq \bbZ/M\bbZ$). Then we have
$$\theta_n(c) \in h_n\cR_n ,$$
where $h_n=|H^1_{(\cF^\ast)_n}(\bbQ, A^{\ast})|$.
\end{theorem}

From this, we obtain the following corollary, which is a special case of \cite[Theorem 2.2.2]{R} and \cite[Corollary 4.4.5]{MR1} (see also Remark \ref{remkappa}).

\begin{corollary} \label{corrubin}
Under the same assumptions in Theorem \ref{mainthm}, we have $c_\bbQ=\theta_1(c) \in H^1_\cF(\bbQ,A)$ and
$$\ord_p(h_1) \leq \ind(c),$$
where $\ord_p(h_1)$ is defined by $h_1=p^{\ord_p(h_1)}$, and 
$$\ind(c)=\sup \{ m \ | \ c_\bbQ \in p^m H^1_{\cF}(\bbQ,A) \}.$$
\end{corollary}

\begin{proof}
Take $n=1$ in Theorem \ref{mainthm}, then we have
$$c_\bbQ=\theta_1(c)\in h_1\cR_1 =h_1 H^1_{\cF}(\bbQ,A).$$
Hence we have the desired inequality $\ord_p(h_1) \leq \ind(c).$
\end{proof}

\section{Algebraic Kolyvagin systems} \label{secKoly}
In this section, we introduce a notion of ``algebraic Kolyvagin systems". The aim of this section is to prove Theorem \ref{thmKS}. Our Kolyvagin systems are defined for a 7-tuple 
$(\cO, \Sigma, H, t, v, u, P)$ satisfying the following:
\begin{itemize}
\item{$\cO$: a commutative ring (with unity),}
\item{$\Sigma$: a countable set,}
\item{$H$: an $\cO$-module,}
\item{$t=\{t_\frq\}_\frq \in \prod_{\frq \in \Sigma}\bbZ_{\geq1}$,}
\item{$v=\{v_\frq\}_\frq \in \prod_{\frq \in \Sigma}\Hom_\cO(H,\cO)$,}
\item{$u=\{u_\frq\}_\frq \in \prod_{\frq \in \Sigma}\Hom_\cO(H,\cO/(t_\frq))$ ($(t_\frq)$ denotes the ideal $t_\frq \cO$),}
\item{$P=\{P_\frq\}_\frq \in \prod_{\frq \in \Sigma}G(\Sigma\setminus \frq)_1$ (we often denote $\Sigma\setminus \{\frq\}$ by $\Sigma\setminus \frq$),}
\end{itemize}
where for any subset $\Sigma' \subset \Sigma$, 
$$G(\Sigma')_i= \mathop{\ilim}_{\frn \in \cN(\Sigma')} \left( I_\frn^i/I_\frn^{i+1} \otimes_\bbZ \cO \right),$$
and where $\cN(\Sigma')=\{ \frn \subset \Sigma'  \ | \ \nu(\frn):= |\frn| < \infty\}$, and $I_\frn$ is the augmentation ideal of $\bbZ[\bigoplus_{\frq \in \frn}\bbZ/t_\frq \bbZ]$. 

Note that $G(\Sigma')_1$ is canonically isomorphic to $\prod_{\frq \in \Sigma'}\cO/(t_\frq)$, 
since 
$$I_\frn / I_\frn^2 \otimes_\bbZ \cO \simeq \bigoplus_{\frq \in \frn}\bbZ/t_\frq \bbZ \otimes_\bbZ \cO \simeq \bigoplus_{\frq \in \frn}\cO/(t_\frq)$$ 
for any $\frn \in \cN(\Sigma')$, where the first isomorphism is induced by the inverse of 
$$\bigoplus_{\frq\in\frn}\bbZ/t_\frq\bbZ \stackrel{\sim}{\longrightarrow} I_\frn/I_\frn^2 \quad ; \quad \sigma \mapsto \sigma-1.$$
So if $\Sigma'' \subset \Sigma'$, then $G(\Sigma'')_1$ is regarded as an $\cO$-submodule 
of $G(\Sigma')_1$, and also a quotient of it. We put
$$G(\Sigma')= \mathop{\ilim}_{\frn \in \cN(\Sigma')} \left(\bigoplus_{i=0}^\infty I_\frn^i/I_\frn^{i+1} \otimes_\bbZ \cO \right).$$
Note that if $\Sigma'' \subset \Sigma'$, then there is a natural map from $G(\Sigma'')$ to $G(\Sigma')$ induced by the inclusion $I_\frn \hookrightarrow I_\frm$, where $\frn \in \cN(\Sigma'')$, $\frm\in\cN(\Sigma')$ such that $\frn \subset \frm$. So any element of $G(\Sigma'')$ is 
naturally regarded as an element of $G(\Sigma')$. 

From now on we fix a 7-tuple $(\cO, \Sigma, H, t, v, u, P)$ satisfying above, and give some more notations for it. We denote simply 
$\cN=\cN(\Sigma)$. If 
$\Sigma' \subset \Sigma$, there is a natural projection map from $G(\Sigma)$ to $G(\Sigma')$, which we denote by 
$(  \cdot  )|_{\Sigma'}$. In particular, for $\frn \in \cN$, which is by definition a subset of $\Sigma$, we denote the 
projection map to $G(\frn)$ by $\pi_\frn$ (namely, $\pi_\frn:=(\cdot)|_\frn : G(\Sigma) \rightarrow G(\frn)$).

If $\frm, \frn \in \cN$, and $\frm \subset \frn$, we denote $\frn/\frm$ instead of the set theoretic notation $\frn \setminus \frm$. If $\frn \in \cN$ and $\frq \in \Sigma$ such that $\frq \notin \frn$, we denote $\frn\frq$ instead of $\frn \cup \frq$.
We also denote $1$ instead of $\emptyset \in \cN$.

For each $\frq \in \Sigma$, fix a generator $x_\frq$ of $G(\frq)_1 (\simeq \cO/(t_\frq))$ (as an $\cO$-module).

\begin{definition} \label{defrec}  
For any $\frq \in \Sigma$, we define an $\cO$-homomorphism
$$\varphi_\frq : H \longrightarrow G(\Sigma)_1$$
by $\varphi_\frq(a)=-u_\frq(a) x_\frq-v_\frq(a) P_\frq$. For $\frn \in \cN$, we denote the composition map $\pi_\frn \circ \varphi_\frq$ by $\varphi_\frq^\frn$.
\end{definition}

Note that if $\frn \in \cN$ and $\frn=\frd \sqcup \frm$, we have $\varphi_\frq^\frn=\varphi_\frq^\frd + \varphi_\frq^\frm$ for any $\frq \in \Sigma$, since $G(\frn)_1 (\simeq \bigoplus_{\frq' \in \frn}\cO/(t_{\frq'})) \simeq G(\frd)_1 \oplus G(\frm)_1$.

\begin{example} \label{ex1}
The setting in \S \ref{statement} fits into this general setting. Use the notations as in \S \ref{statement}, take $(\cO, \Sigma, H, t, v, u, P)$ as follows:
\begin{itemize}
\item{$\cO=\bbZ/M\bbZ$,}
\item{$\Sigma$: as in \S \ref{statement},}
\item{$H=\bigcup_{n \in \cN(\Sigma)}H^1_{\cF^n}(\bbQ,A)$,} 
\item{$t_\ell$: the maximal $p$-power dividing $\ell-1$,}
\item{$v_\ell$: as in \S \ref{statement},}
\item{$u_\ell$: as in \S \ref{statement},}
\item{$P=(P_\ell(\Fr_\ell) \otimes 1)_\ell \in \prod_{\ell \in \Sigma}\mathop{\ilim}_{n\in\cN(\Sigma), \ell|n}(I_{n/\ell}/I_{n/\ell}^2 \otimes \bbZ/M\bbZ)=\prod_{\ell \in \Sigma} G(\Sigma\setminus\ell)_1$.}
\end{itemize}
If we set $x_\ell=(\sigma_\ell-1)\otimes 1$, then $\varphi_\ell$ in the above definition is the same one in \S \ref{statement}.

\end{example}

Now, for $r \in \bbZ_{\geq1}$, we define algebraic Kolyvagin systems of ``rank $r$". Recall that for $\frn \in \cN$, $\nu(\frn)=|\frn|$, 
and $G(\frn)_{\nu(\frn)}=I_\frn^{\nu(\frn)}/I_\frn^{\nu(\frn)+1}\otimes_\bbZ \cO$. In what follows, for any $\cO$-module $G$, we denote $\left(\bigwedge_\cO^rH \right)\otimes_\cO G$ by $\bigwedge^rH\otimes_\cO G$ for simplicity. 

By the construction in \S \ref{secalg}, for every $\frq \in \Sigma$ and $\frn \in \cN$, $v_\frq \in \Hom_\cO(H,\cO)$ induces the map 
$$v_\frq : \bigwedge^rH \otimes_\cO G(\frn)_{\nu(\frn)} \longrightarrow \bigwedge^{r-1}H \otimes_\cO G(\frn)_{\nu(\frn)}.$$
Similarly, $u_\frq \in \Hom_\cO(H,\cO/(t_\frq))$ induces the map 
$$u_\frq : \bigwedge^rH \otimes_\cO G(\frn)_{\nu(\frn)} \longrightarrow \bigwedge^{r-1}H \otimes_\cO G(\frn)_{\nu(\frn)}\otimes_\cO \cO/(t_\frq),$$
and $\varphi_\frq \in \Hom_\cO(H,G(\Sigma)_1)$ induces the map 
$$\varphi_\frq :  \bigwedge^rH \otimes_\cO G(\frn)_{\nu(\frn)} \longrightarrow \bigwedge^{r-1}H \otimes_\cO G(\Sigma)_{\nu(\frn) +1}.$$

\begin{definition} \label{defkoly}
A collection 
$$\{ \kappa_\frn \in \bigwedge^rH \otimes_\cO G(\frn)_{\nu(\frn)} \ | \ \frn \in \cN \}$$
is a Kolyvagin system of rank $r$ if the following axioms (K1)-(K4) are satisfied:\\
(K1) if $\frq \in \Sigma \setminus \frn$, then $v_\frq(\kappa_\frn)=0$,\\
(K2) if $\frq \in \frn$, then $u_\frq(\kappa_\frn)=0$,\\
(K3) if $\frq \in \frn$, then $v_\frq(\kappa_\frn)=\varphi_\frq(\kappa_{\frn/\frq})$,\\
(K4) if $\frq \in \frn$, then $\pi_{\frn/\frq}(\kappa_\frn)=0$.

We denote the $\cO$-module consisting of all Kolyvagin systems of rank $r$ by KS$_r$. This is an $\cO$-submodule of $\prod_{\frn \in \cN}\bigwedge^rH\otimes_\cO G(\frn)_{\nu(\frn)}$.
\end{definition}

We will see that our Kolyvagin systems generalize the notion of original Kolyvagin systems in \cite{MR1} (see Proposition \ref{propKS}).

We will define other three algebraic Kolyvagin systems, in Definitions \ref{defTKS}, \ref{defPKS}, and \ref{defDKS}, which we call 
$\theta$-Kolyvagin systems, pre-Kolyvagin systems, and derived-Kolyvagin systems respectively. The $\cO$-module consisting of 
all $\theta$-Kolyvagin systems (resp. pre-Kolyvagin systems, resp. derived-Kolyvagin systems) of rank $r$ is denoted by ${\rm{TKS}}_r$ (resp. 
${\rm{PKS}}_r$, resp. ${\rm{DKS}}_r$).

The following definition is due to \cite[Definition 6.1]{MR2}.

\begin{definition} \label{defd}
Let $\frn \in \cN$ and $\frd \subset \frn$. When $\frd \neq 1$, define
$$\cD_{\frn, \frd}=\left|
\begin{array}{ccccc}
	-\pi_{\frn/\frd}(P_{\frq_1}) &-\pi_{\frq_2}(P_{\frq_1}) &\cdots& &-\pi_{\frq_\nu}(P_{\frq_1}) \\ 
	-\pi_{\frq_1}(P_{\frq_2}) &-\pi_{\frn/\frd}(P_{\frq_2}) &-\pi_{\frq_3}(P_{\frq_2}) & \cdots & -\pi_{\frq_\nu}(P_{\frq_2}) \\
	\vdots & -\pi_{\frq_2}(P_{\frq_3}) & \ddots & &\vdots\\
	\vdots & \vdots & &\ddots &\vdots\\
	-\pi_{\frq_1}(P_{\frq_\nu}) &-\pi_{\frq_2}(P_{\frq_\nu}) &\cdots & &-\pi_{\frn/\frd}(P_{\frq_\nu})
\end{array}
\right| \in G(\frn)_{\nu(\frd)},$$
where $\{ \frq_1, \ldots, \frq_\nu \}=\frd$ ($\nu=\nu(\frd)$). When $\frd=1$, define 
$$\cD_{\frn, 1}=1 \in \cO=G(\frn)_0.$$
Note that $\cD_{\frn,\frd}$ does not depend on the choice of the order $\frq_1,\ldots, \frq_\nu$ of the elements of $\frd$.

We put 
$$\cD_\frd=\pi_\frd(\cD_{\frn, \frd}) =\left|
\begin{array}{ccccc}
	0 &-\pi_{\frq_2}(P_{\frq_1}) &\cdots & &-\pi_{\frq_\nu}(P_{\frq_1}) \\ 
	-\pi_{\frq_1}(P_{\frq_2}) & 0 & -\pi_{\frq_3}(P_{\frq_2}) & \cdots & -\pi_{\frq_\nu}(P_{\frq_2}) \\
	\vdots & -\pi_{\frq_2}(P_{\frq_3}) & \ddots & &\vdots\\
	\vdots & \vdots & &\ddots &\vdots\\
	-\pi_{\frq_1}(P_{\frq_\nu}) &-\pi_{\frq_2}(P_{\frq_\nu}) &\cdots & & 0
\end{array}
\right| \in G(\frd)_{\nu(\frd)}.$$
Clearly, $\cD_\frd$ does not depend on $\frn$. 

\end{definition}

\begin{definition} \label{defTKS}
A collection 
$$\{ \theta_\frn \in \bigwedge^rH \otimes_\cO G(\frn)_{\nu(\frn)} \ | \ \frn \in \cN \}$$
is a $\theta$-Kolyvagin system of rank $r$ if the following axioms (TK1)-(TK4) are satisfied:\\
(TK1) if $\frq \in \Sigma \setminus \frn$, then $v_\frq(\theta_\frn)=0$,\\
(TK2) if $\frq \in \frn$, then $u_\frq(\sum_{\frd \subset \frn}\theta_\frd \cD_{\frn,\frn/\frd})=0$,\\
(TK3) if $\frq \in \frn$, then $v_\frq(\sum_{\frd \subset \frn}(-1)^{\nu(\frn/\frd)}\pi_\frd(\theta_\frn))=\varphi_\frq(\sum_{\frd \subset \frn/\frq}(-1)^{\nu(\frn/\frd \frq)}\pi_\frd(\theta_{\frn/\frq}))$,\\
(TK4) if $\frq \in \frn$, then $\pi_{\frn/\frq}(\theta_\frn)=\theta_{\frn/\frq}\cdot \pi_{\frn/\frq}(P_\frq)$.
\end{definition}

\begin{definition} \label{defPKS}
A collection 
$$\{ \pkappa_\frn \in \bigwedge^rH \otimes_\cO G(\Sigma)_{\nu(\frn)} \ | \ \frn \in \cN \}$$
is a pre-Kolyvagin system of rank $r$ if the following axioms (PK1)-(PK5) are satisfied:\\
(PK1) if $\frq \in \Sigma \setminus \frn$, then $v_\frq(\pkappa_\frn)=0$,\\
(PK2) if $\frq \in \frn$, then $u_\frq(\sum_{\frd \subset \frn}(-1)^{\nu(\frn/\frd)}\pi_\frn(\pkappa_\frd)\prod_{\frq' \in \frn/\frd}\pi_{\frn/\frq'}(P_{\frq'}))=0$,\\
(PK3) if $\frq \in \frn$, then $v_\frq(\pkappa_\frn)=\varphi_\frq(\pkappa_{\frn/\frq})$,\\
(PK4) if $\frq \in \frn$, then $\pkappa_\frn|_{\Sigma\setminus\frq}=\pkappa_{\frn/\frq}|_{\Sigma\setminus\frq}\cdot P_\frq$,\\
(PK5) $\pkappa_\frn=\sum_{\frd \subset \frn}\pi_\frn(\pkappa_\frd)\prod_{\frq \in \frn/\frd}P_\frq|_{\Sigma\setminus\frn}$.

\end{definition}

\begin{definition} \label{defDKS}
A collection 
$$\{ \kappa'_\frn \in \bigwedge^rH \otimes_\cO G(\frn)_{\nu(\frn)} \ | \ \frn \in \cN \}$$
is a derived-Kolyvagin system of rank $r$ if the following axioms (DK1)-(DK4) are satisfied:\\
(DK1) if $\frq \in \Sigma \setminus \frn$, then $v_\frq(\kappa'_\frn)=0$,\\
(DK2) if $\frq \in \frn$, then $u_\frq(\sum_{\frd \subset \frn}\kappa'_\frd \cD_{\frn/\frd})=0$,\\
(DK3) if $\frq \in \frn$, then $v_\frq(\kappa'_\frn)=\varphi_\frq(\kappa'_{\frn/\frq})$,\\
(DK4) if $\frq \in \frn$, then $\pi_{\frn/\frq}(\kappa'_\frn)=0$.

\end{definition}

\begin{remark}
The notion of ``pre-Kolyvagin systems" first appeared in \cite[Definition 6.2]{MR2}. Note that the notion which generalizes pre-Kolyvagin systems in \cite{MR2} 
is what we call $\theta$-Kolyvagin systems in this paper. We use the terminology ``pre-Kolyvagin system" for a different system.
\end{remark}

Next we define morphisms between these Kolyvagin systems. In the following definition, the meaning of the subscript of $F_{PT}$ is ``from pre-Kolyvagin systems to $\theta$-Kolyvagin systems", and that of $F_{PK}$, $F_{TK}$, etc. are similar (see Theorem \ref{thmKS}).

\begin{definition} \label{defmor}  
We define homomorphisms $F_{PT}$ and $F_{PK}$ from $\prod_{\frn \in \cN}\bigwedge^rH\otimes_\cO G(\Sigma)_{\nu(\frn)}$ to $\prod_{\frn \in \cN}\bigwedge^rH\otimes_\cO G(\frn)_{\nu(\frn)}$ by 
$$F_{PT}(\{a_\frn\}_\frn)=\left\{\pi_\frn(a_\frn)\right\}_\frn,$$
$$F_{PK}(\{a_\frn\}_\frn)=\left\{\sum_{\frd \subset \frn}(-1)^{\nu(\frn/\frd)}\pi_\frn(a_\frd)\prod_{\frq \in \frn/\frd}\pi_{\frn/\frq}(P_\frq)\right\}_\frn.$$
We define endomorphisms $F_{TK}, F_{TD}$, and $F_{DK}$ of $\prod_{\frn \in \cN}\bigwedge^rH\otimes_\cO G(\frn)_{\nu(\frn)}$ by
$$F_{TK}(\{a_\frn\}_\frn)=\left\{\sum_{\frd \subset \frn}a_\frd \cD_{\frn,\frn/\frd}\right\}_\frn,$$
$$F_{TD}(\{a_\frn\}_\frn)=\left\{\sum_{\frd \subset \frn}(-1)^{\nu(\frn/\frd)}a_\frd \prod_{\frq \in \frn/\frd}\pi_{\frd}(P_\frq)\right\}_\frn,$$
$$F_{DK}(\{a_\frn\}_\frn)=\left\{\sum_{\frd \subset \frn}a_\frd \cD_{\frn/\frd} \right\}_\frn.$$
\end{definition}

\begin{proposition} \label{propinj}     
$F_{TK}, F_{TD},$ and $F_{DK}$ are injective.
\end{proposition}

\begin{proof}
We only show for $F_{TK}$. One can show the injectivity for the others by the same method. Suppose $\{a_\frn\}_\frn \in \Ker F_{TK}$, i.e. 
$$\sum_{\frd \subset \frn}a_\frd \cD_{\frn,\frn/\frd}=0$$
for all $\frn \in \cN$. We show by induction on $\nu(\frn)$ that $a_\frn=0$. When $\nu(\frn)=0$, i.e. $\frn=1$, we have $\sum_{\frd \subset \frn}a_\frd \cD_{\frn,\frn/\frd}=a_1$ and this is $0$ by the assumption. When $\nu(\frn)>0$, by the 
inductive hypothesis we have
$$\sum_{\frd \subset \frn}a_\frd \cD_{\frn,\frn/\frd}=a_\frn \cD_{\frn,1}=a_\frn.$$
Since the left hand side is $0$ by the assumption that $F_{TK}(\{a_\frn\}_\frn)=0$, we get $a_\frn=0$.
\end{proof}

We define the following useful operator $s_{\frm,\frn}$. 

\begin{definition} \label{defs}   
For $\frn, \frm \in \cN$ such that $\frn \subset \frm$, we define an operator $s_{\frm,\frn}$ on $G(\frm)$ by 
$$s_{\frm,\frn}(g)=\sum_{\frd \subset \frn}(-1)^{\nu(\frd)}\pi_{\frm/\frd}(g).$$
This is an $\cO$-endomorphism of $G(\frm)$. 
When $\frm=\frn$, put $s_\frn=s_{\frn,\frn}$.

\end{definition}

\begin{lemma} \label{lem1}
Let $\cM$ be an $\cO$-module, and $\frn,\frm \in \cN$ such that $\frn \subset \frm$. We regard $s_{\frm,\frn}$ as an operator on $\cM\otimes_\cO G(\frm)$. Then we have the following:\\
{\rm{(i)}} $$s_{\frm,\frn}(\cM\otimes_\cO G(\frm)) \subset \cM\otimes_\cO \left( \prod_{\frq \in \frn}x_\frq \right) ,$$
where $x_\frq$ is the fixed generator of $G(\frq)_1$ and $(\prod_{\frq \in \frn}x_\frq)$ is the (principal) ideal of $G(\frm)$ generated by $\prod_{\frq \in \frn}x_\frq$. 

In particular, we have $\pi_{\frn/\frq}\circ s_{\frm,\frn}=0$ for all $\frq\in\frn$.
\\
{\rm{(ii)}} If $\frd, \frn \in \cN$ and $\frd \subset \frn$, 
and $g \in \cM\otimes_\cO G(\frd)_{\nu(\frd)}$, $h \in G(\frn)_{\nu(\frn/\frd)}$, then we have 
$$s_\frn(gh)=s_\frd(g)s_{\frn,\frn/\frd}(h).$$
\end{lemma}

\begin{proof}
(i) Suppose $\frn=\{ \frq_1,\ldots,\frq_\nu \}$ ($\nu=\nu(\frn)$). Take any generator of $\cM\otimes_\cO G(\frm)$, and write it as follows:
$$\sum_{\alpha}m_\alpha \otimes g_\alpha x_{\frq_1}^{\alpha_1}\cdots x_{\frq_\nu}^{\alpha_\nu},$$
where $\alpha$ runs over $\bbZ_{\geq0}^\nu$, $m_\alpha \in\cM$, and $g_\alpha \in G(\frm/\frn)$. 
Put $\frd_\alpha=\{ \frq_i \in \frn \ | \ \alpha_i=0 \}$. We have 
\begin{eqnarray}
s_{\frm,\frn}\left(\sum_{\alpha}m_\alpha \otimes g_\alpha x_{\frq_1}^{\alpha_1}\cdots x_{\frq_\nu}^{\alpha_\nu}\right) &=& \sum_{\alpha}m_\alpha\otimes\left( \sum_{\frd \subset \frn}(-1)^{\nu(\frd)}g_\alpha \pi_{\frm/\frd}(x_{\frq_1}^{\alpha_1}\cdots x_{\frq_\nu}^{\alpha_\nu}) \right) \nonumber \\
&=& \sum_{\alpha}m_\alpha \otimes \left(\sum_{\frd \subset \frd_\alpha}(-1)^{\nu(\frd)}g_\alpha x_{\frq_1}^{\alpha_1}\cdots x_{\frq_\nu}^{\alpha_\nu} \right) \nonumber 
\end{eqnarray}
(note that since $g_\alpha \in G(\frm/\frn)$, we have $\pi_{\frm/\frd}(g_\alpha)=g_\alpha$ for any $\frd \subset \frn$). If $\nu(\frd_\alpha)>0$, then we have 
$$\sum_{\frd \subset \frd_\alpha}(-1)^{\nu(\frd)}=(1-1)^{\nu(\frd_\alpha)}=0.$$
Hence we have 
$$s_{\frm,\frn}\left(\sum_{\alpha}m_\alpha \otimes g_\alpha x_{\frq_1}^{\alpha_1}\cdots x_{\frq_\nu}^{\alpha_\nu}\right)=
\sum_{\alpha, \alpha_i \geq 1}m_\alpha \otimes g_\alpha x_{\frq_1}^{\alpha_1}\cdots x_{\frq_\nu}^{\alpha_\nu}  \in \cM\otimes_\cO \left( \prod_{\frq \in \frn}x_\frq \right).$$

(ii) Suppose $\frd=\{ \frq_1,\ldots,\frq_\mu \}$, and $\frn/\frd=\{ \frq'_1,\ldots,\frq'_\nu \}$ ($\mu=\nu(\frd),$ $\nu=\nu(\frn/\frd)$). Write $g$ and $h$ as follows:
$$g=\sum_{ |\alpha|=\mu}m_\alpha \otimes x_{\frq_1}^{\alpha_1}\cdots x_{\frq_\mu}^{\alpha_\mu},$$
$$h=\sum_{ |\beta+\gamma|=\nu}a_{\beta,\gamma} x_{\frq_1}^{\beta_1}\cdots x_{\frq_\mu}^{\beta_\mu} x_{\frq'_1}^{\gamma_{1}}\cdots x_{\frq'_\nu}^{\gamma_{\nu}},$$
where $m_\alpha \in \cM$ and $a_{\beta,\gamma}\in \cO$ ($|\alpha|$ means $\alpha_1+\cdots +\alpha_\mu$, and $|\beta+\gamma|$ is similar). As in the proof of (i), we have 
$$s_\frd(g)=m_{(1,\ldots,1)}\otimes x_{\frq_1}\cdots x_{\frq_\mu},$$
$$s_{\frn,\frn/\frd}(h)=a_{(0,\ldots,0),(1,\ldots,1)}x_{\frq'_1}\cdots x_{\frq'_\nu},$$
and 
$$s_\frn(gh)=a_{(0,\ldots,0),(1,\ldots,1)}m_{(1,\ldots,1)}\otimes x_{\frq_1}\cdots x_{\frq_\mu}x_{\frq'_1}\cdots x_{\frq'_\nu}.$$
Hence we have 
$$s_\frn(gh)=s_\frd(g)s_{\frn,\frn/\frd}(h).$$
\end{proof}

\begin{corollary} \label{cor1}
Let $\frn, \frm \in \cN$ such that $\frn \subset \frm$, and $g \in \cM\otimes_\cO G(\frm)_{\nu(\frn)}$. If $\pi_{\frm/\frq}(g)=0$ for 
every $\frq \in \frn$, then we have
$$g \in \cM\otimes_\cO <\prod_{\frq \in \frn}x_\frq>_\cO,$$
where $<\prod_{\frq \in \frn}x_\frq>_\cO$ is the $\cO$-submodule of $G(\frm)$ generated by $\prod_{\frq \in \frn}x_\frq$. 

In particular, we have $g \in \cM\otimes_\cO G(\frn)$.
\end{corollary}

\begin{proof}
Suppose $\frn=\{ \frq_1,\ldots,\frq_\nu\}$ ($\nu=\nu(\frn)$). Write $g$ as
$$g=\sum_j\sum_{\alpha}m_j \otimes g_{j,\alpha} x_{\frq_1}^{\alpha_1}\cdots x_{\frq_\nu}^{\alpha_\nu},$$
where $m_j \in\cM$, and $g_{j,\alpha} \in G(\frm/\frn)$. As in the proof of Lemma \ref{lem1}, we have 
$$s_{\frm,\frn}\left( \sum_j\sum_{\alpha}m_j \otimes g_{j,\alpha} x_{\frq_1}^{\alpha_1}\cdots x_{\frq_\nu}^{\alpha_\nu} \right)=\sum_j\sum_{\alpha, \alpha_i \geq1}m_j \otimes g_{j,\alpha} x_{\frq_1}^{\alpha_1}\cdots x_{\frq_\nu}^{\alpha_\nu}.$$
Since $\pi_{\frm/\frq}(g)=0$ for every $\frq \in \frn$ by the assumption, we have $s_{\frm,\frn}(g)=g$ (by the definition of $s_{\frm,\frn}$). 
Hence we have 
$$g=\sum_j\sum_{\alpha, \alpha_i \geq1}m_j \otimes g_{j,\alpha} x_{\frq_1}^{\alpha_1}\cdots x_{\frq_\nu}^{\alpha_\nu}.$$
Since $g\in\cM\otimes_\cO G(\frm)_\nu$ ($g$ is ``homogeneous of degree $\nu$"), each $\alpha_i$ must be equal to $1$, and hence the right hand side must be in $\cM\otimes_\cO <\prod_{\frq \in \frn}x_\frq>_\cO$. 
\end{proof}

\begin{lemma} \label{lem2}
If $\{ \pkappa_\frn \in \bigwedge^rH \otimes_\cO G(\Sigma)_{\nu(\frn)} \ | \ \frn \in \cN \}$ satisfies {\rm(PK4)}, then 
we have the following: if $\frn \subset \frm$, then for every $\frq \in \frn$, we have 
$$\pi_{\frm/\frq}( \sum_{\frd \subset \frn}(-1)^{\nu(\frn/\frd)}\pi_\frm(\pkappa_\frd)\prod_{\frq' \in \frn/\frd}\pi_{\frm/\frq'}(P_{\frq'})  )=0.$$
\end{lemma}

\begin{proof}
\begin{eqnarray}
& & \pi_{\frm/\frq}( \sum_{\frd \subset \frn}(-1)^{\nu(\frn/\frd)}\pi_\frm(\pkappa_\frd)\prod_{\frq' \in \frn/\frd}\pi_{\frm/\frq'}(P_{\frq'}) ) \nonumber \\
&=& \pi_{\frm/\frq}( \sum_{\frd \subset \frn/\frq}(-1)^{\nu(\frn/\frd)}\pi_\frm(\pkappa_\frd)\prod_{\frq' \in \frn/\frd}\pi_{\frm/\frq'}(P_{\frq'})  + \sum_{\frd \subset \frn/\frq}(-1)^{\nu(\frn/\frd\frq)}\pi_\frm(\pkappa_{\frd\frq})\prod_{\frq'' \in \frn/\frd\frq}\pi_{\frm/\frq''}(P_{\frq''})  ) \nonumber \\
&=& \sum_{\frd \subset \frn/\frq}(-1)^{\nu(\frn/\frd)}\pi_{\frm/\frq}( \pkappa_\frd\prod_{\frq' \in \frn/\frd}\pi_{\frm/\frq'}(P_{\frq'}) )  + \sum_{\frd \subset \frn/\frq}(-1)^{\nu(\frn/\frd\frq)}\pi_{\frm/\frq}( \pkappa_{\frd\frq}\prod_{\frq'' \in \frn/\frd\frq}\pi_{\frm/\frq''}(P_{\frq''}) ) \nonumber \\
&=& \sum_{\frd \subset \frn/\frq}(-1)^{\nu(\frn/\frd)}\pi_{\frm/\frq}( \pkappa_\frd\prod_{\frq' \in \frn/\frd}\pi_{\frm/\frq'}(P_{\frq'}) ) + \sum_{\frd \subset \frn/\frq}(-1)^{\nu(\frn/\frd\frq)}\pi_{\frm/\frq}( \pkappa_\frd\prod_{\frq'' \in \frn/\frd}\pi_{\frm/\frq''}(P_{\frq''})  ) \nonumber \\
&=& 0, \nonumber
\end{eqnarray}
where the third equality follows from (PK4).
\end{proof}

\begin{proposition} \label{prop1} 
{\rm(i)} {\rm{(PK5)}} is equivalent to the following: \\
{\rm{(PK5)$'$}} if $\frn \subset \frm$, then $\pkappa_\frn=\sum_{\frd \subset \frn}\pi_\frm(\pkappa_\frd)\prod_{\frq \in \frn/\frd}P_\frq|_{\Sigma \setminus \frm}.$\\
{\rm(ii)} If $\{ \pkappa_\frn \in \bigwedge^rH \otimes_\cO G(\Sigma)_{\nu(\frn)} \ | \ \frn \in \cN \}$ satisfies {\rm(PK4)}, then 
we have the following: if $\frn \subset \frm$, then we have an equality in $\bigwedge^rH\otimes_\cO G(\frm)_{\nu(\frn)}$:
$$\sum_{\frd \subset \frn}(-1)^{\nu(\frn/\frd)}\pi_\frm(\pkappa_\frd)\prod_{\frq \in \frn/\frd}\pi_{\frm/\frq}(P_\frq)=
\sum_{\frd \subset \frn}(-1)^{\nu(\frn/\frd)}\pi_\frn(\pkappa_\frd)\prod_{\frq \in \frn/\frd}\pi_{\frn/\frq}(P_\frq).$$
\end{proposition}

\begin{proof}
(i) One sees immediately that (PK5)$'$ implies (PK5) (take $\frm=\frn$ in (PK5)$'$, this is (PK5)). Suppose (PK5) and we show (PK5)$'$ by induction 
on $\nu(\frn)$. When $\nu(\frn)=0$, i.e. $\frn=1$, we have 
$$\pkappa_1=\pi_\frm(\pkappa_1)$$
for any $\frm$ since $\pkappa_1 \in \bigwedge^rH\otimes_\cO G(\Sigma)_0=\bigwedge^rH$, and we have 
$$\sum_{\frd \subset 1}\pi_\frm(\pkappa_\frd)\prod_{\frq \in 1/\frd}P_\frq|_{\Sigma \setminus \frm}=\pi_\frm(\pkappa_1)$$
so (PK5)$'$ is satisfied in this case. When $\nu(\frn)>0$, we prove (PK5)$'$ by induction on $\nu(\frm/\frn)$. 
When $\nu(\frm/\frn)=0$, i.e. $\frm=\frn$, there is nothing to prove because it is (PK5). When $\nu(\frm/\frn) >0$, take any $\frq \in \frm/\frn$. We have for any $\frd \subset \frn$
\begin{eqnarray}
\pi_\frm(\pkappa_\frd)=\sum_{\frc \subset \frd}\pi_{\frm/\frq}(\pkappa_\frc)\prod_{\frq' \in \frd/\frc}\pi_\frq(P_{\frq'}) \label{eq1}.
\end{eqnarray}
To see this, if $\frd \neq \frn$ we get this equality by the inductive hypothesis on $\nu(\frn)$ 
(replace $\frn$, $\frm$ in (PK5)$'$ by $\frd$,  $\frm/\frq$ respectively then apply $\pi_\frm$). 
If $\frd=\frn$ we get the equality by the inductive hypothesis on $\nu(\frm/\frn)$ (replace 
$\frm$ in (PK5)$'$ by $\frm/\frq$ then apply $\pi_\frm$).

Hence we have 
\begin{eqnarray}
\sum_{\frd \subset \frn}\pi_\frm(\pkappa_\frd)\prod_{\frq' \in \frn/\frd}P_{\frq'}|_{\Sigma \setminus \frm} &=& 
\sum_{\frd \subset \frn}\sum_{\frc \subset \frd}\pi_{\frm/\frq}(\pkappa_\frc)\prod_{\frq'' \in \frd/\frc}\pi_\frq(P_{\frq''})\prod_{\frq' \in \frn/\frd}P_{\frq'}|_{\Sigma \setminus \frm} \nonumber \\
&=& \sum_{\frd \subset \frn}\pi_{\frm/\frq}(\pkappa_\frd)\prod_{\frq' \in \frn/\frd}P_{\frq'}|_{\Sigma \setminus (\frm/\frq)} \nonumber \\
&=& \pkappa_\frn, \nonumber
\end{eqnarray}
where the first equality is obtained by (\ref{eq1}), and the second is by the direct computation (note that $P_{\frq'}|_{\Sigma \setminus (\frm/\frq)}=P_{\frq'}|_{\Sigma \setminus \frm}+\pi_\frq(P_{\frq'})$), and the last is by the inductive hypothesis on $\nu(\frm/\frn)$ 
(replace $\frm$ in (PK5)$'$ by $\frm/\frq$). This completes the proof of (i).

(ii) From Lemma \ref{lem2} and Corollary \ref{cor1}, we have 
$$\sum_{\frd \subset \frn}(-1)^{\nu(\frn/\frd)}\pi_\frm(\pkappa_\frd)\prod_{\frq \in \frn/\frd}\pi_{\frm/\frq}(P_\frq) \in \bigwedge^rH\otimes_\cO G(\frn),$$
so the left hand side does not change when we apply $\pi_\frn$. Hence we have 
\begin{eqnarray}
\sum_{\frd \subset \frn}(-1)^{\nu(\frn/\frd)}\pi_\frm(\pkappa_\frd)\prod_{\frq \in \frn/\frd}\pi_{\frm/\frq}(P_\frq)&=&
\pi_\frn \left(\sum_{\frd \subset \frn}(-1)^{\nu(\frn/\frd)}\pi_\frm(\pkappa_\frd)\prod_{\frq \in \frn/\frd}\pi_{\frm/\frq}(P_\frq) \right) \nonumber \\
&=& \sum_{\frd \subset \frn}(-1)^{\nu(\frn/\frd)}\pi_\frn(\pkappa_\frd)\prod_{\frq \in \frn/\frd}\pi_{\frn/\frq}(P_\frq). \nonumber
\end{eqnarray}

\end{proof}

\begin{proposition} \label{propd}
Suppose $\frd, \frn \in \cN$ and $\frd \subset \frn$. \\
{\rm{(i)}} If $\frq \in \frd$, then $\pi_{\frn/\frq}(\cD_{\frn, \frd})=-\cD_{\frn/\frq,\frd/\frq}\cdot\pi_{\frn/\frd}(P_\frq).$ \\
{\rm{(ii)}} If $\frq \in \frn/\frd$, then $\pi_{\frn/\frq}(\cD_{\frn,\frd})=\cD_{\frn/\frq,\frd}.$ \\
{\rm{(iii)}} $s_{\frn,\frd}(\cD_{\frn,\frd})=\cD_\frd.$
\end{proposition}

\begin{proof}

(i) Suppose $\frd=\{ \frq_1,\ldots, \frq_\nu \}$ and $\frq=\frq_\nu$. By the definition of $\cD_{\frn,\frd}$ (see Definition \ref{defd}), we have
\begin{eqnarray}
\pi_{\frn/\frq} \left( \cD_{\frn, \frd} \right) &=& \pi_{\frn/\frq}\left(\left|
\begin{array}{ccccc}
	-\pi_{\frn/\frd}(P_{\frq_1}) &-\pi_{\frq_2}(P_{\frq_1}) &\cdots& &-\pi_{\frq_\nu}(P_{\frq_1}) \\ 
	-\pi_{\frq_1}(P_{\frq_2}) &-\pi_{\frn/\frd}(P_{\frq_2}) &-\pi_{\frq_3}(P_{\frq_2}) & \cdots & -\pi_{\frq_\nu}(P_{\frq_2}) \\
	\vdots & -\pi_{\frq_2}(P_{\frq_3}) & \ddots & &\vdots\\
	\vdots & \vdots & &\ddots &\vdots\\
	-\pi_{\frq_1}(P_{\frq_\nu}) &-\pi_{\frq_2}(P_{\frq_\nu}) &\cdots & &-\pi_{\frn/\frd}(P_{\frq_\nu})
\end{array}
\right|\right) \nonumber \\
&=& \left|
\begin{array}{ccccc}
	-\pi_{\frn/\frd}(P_{\frq_1}) &-\pi_{\frq_2}(P_{\frq_1}) &\cdots& &0 \\ 
	-\pi_{\frq_1}(P_{\frq_2}) &-\pi_{\frn/\frd}(P_{\frq_2}) &-\pi_{\frq_3}(P_{\frq_2}) & \cdots & 0\\
	\vdots & -\pi_{\frq_2}(P_{\frq_3}) & \ddots & &\vdots\\
	\vdots & \vdots & &\ddots &0\\
	-\pi_{\frq_1}(P_{\frq_\nu}) &-\pi_{\frq_2}(P_{\frq_\nu}) &\cdots & &-\pi_{\frn/\frd}(P_{\frq_\nu})
\end{array}
\right| \nonumber \\
&=& -\cD_{\frn/\frq,\frd/\frq}\cdot\pi_{\frn/\frd}(P_\frq). \nonumber
\end{eqnarray}

(ii) Suppose $\frd=\{ \frq_1,\ldots, \frq_\nu \}$. By the definition of $\cD_{\frn,\frd}$, we have
\begin{eqnarray}
\pi_{\frn/\frq} \left( \cD_{\frn, \frd} \right) &=& \pi_{\frn/\frq}\left(\left|
\begin{array}{ccccc}
	-\pi_{\frn/\frd}(P_{\frq_1}) &-\pi_{\frq_2}(P_{\frq_1}) &\cdots& &-\pi_{\frq_\nu}(P_{\frq_1}) \\ 
	-\pi_{\frq_1}(P_{\frq_2}) &-\pi_{\frn/\frd}(P_{\frq_2}) &-\pi_{\frq_3}(P_{\frq_2}) & \cdots & -\pi_{\frq_\nu}(P_{\frq_2}) \\
	\vdots & -\pi_{\frq_2}(P_{\frq_3}) & \ddots & &\vdots\\
	\vdots & \vdots & &\ddots &\vdots\\
	-\pi_{\frq_1}(P_{\frq_\nu}) &-\pi_{\frq_2}(P_{\frq_\nu}) &\cdots & &-\pi_{\frn/\frd}(P_{\frq_\nu})
\end{array}
\right|\right) \nonumber \\
&=& \left|
\begin{array}{ccccc}
	-\pi_{\frn/\frd\frq}(P_{\frq_1}) &-\pi_{\frq_2}(P_{\frq_1}) &\cdots& &-\pi_{\frq_\nu}(P_{\frq_1}) \\ 
	-\pi_{\frq_1}(P_{\frq_2}) &-\pi_{\frn/\frd\frq}(P_{\frq_2}) &-\pi_{\frq_3}(P_{\frq_2}) & \cdots & -\pi_{\frq_\nu}(P_{\frq_2}) \\
	\vdots & -\pi_{\frq_2}(P_{\frq_3}) & \ddots & &\vdots\\
	\vdots & \vdots & &\ddots &\vdots\\
	-\pi_{\frq_1}(P_{\frq_\nu}) &-\pi_{\frq_2}(P_{\frq_\nu}) &\cdots & &-\pi_{\frn/\frd\frq}(P_{\frq_\nu})
\end{array}
\right| \nonumber \\
&=& \cD_{\frn/\frq,\frd}. \nonumber
\end{eqnarray}

(iii) As in the proof of Lemma \ref{lem1} (i), $s_{\frn,\frd}$ eliminates all the terms other than ``$\prod_{\frq \in \frd}x_\frq$-terms". When we expand the determinant $\cD_{\frn,\frd}$, the sum of its ``$\prod_{\frq \in \frd}x_\frq$-terms" is equal to $\cD_\frd$. Hence we have $s_{\frn,\frd}(\cD_{\frn,\frd})=\cD_{\frd}$.
\end{proof}

\begin{theorem} \label{thmKS}   
The following diagram is commutative and all the morphisms are isomorphisms:
\[\xymatrix{
{\rm{PKS}}_r \ar[d]_{F_{PK}} \ar[r]^{F_{PT}} & {\rm{TKS}}_r \ar[d]^{F_{TD}} \ar[dl]_{F_{TK}} \\
{\rm{KS}}_r  & {\rm{DKS}}_r. \ar[l]^{F_{DK}} \\
}\]
\end{theorem}

\begin{remark}
It is shown in \cite[Proposition 6.5]{MR2} that $F_{TK}$ induces isomorphism ${\rm{TKS}}_r \simeq {\rm{KS}}_r$ in a special case. Theorem \ref{thmKS} is a generalization of it.
\end{remark}

\begin{proof}

The strategy of the proof is as follows. The proof is divided into 5 steps. 

In Steps 1, 2, and 3, we show that $F_{PK}$, $F_{TD}$, and $F_{TK}$ are isomorphisms respectively. 

In Steps 4 and 5, we show that $F_{DK}\circ F_{TD}=F_{TK}$ and $F_{TK}\circ F_{PT}=F_{PK}$ respectively. 

By Steps 1, 3, and 5 and Proposition \ref{propinj}, we see that $F_{PT}$ is an isomorphism. By Steps 2, 3 and 4, we see that $F_{DK}$ is an 
isomorphism. Hence by all the steps, we complete the proof. 
\\
\textbf{Step 1.} We show that $F_{PK}$ is an isomorphism. Step 1 is divided into 3 steps. 

In Step 1.1, we show $F_{PK}({\rm{PKS}}_r) \subset {\rm{KS}}_r.$ 

In Step 1.2, we construct the inverse $G_{PK}$ of $F_{PK}$ and show $G_{PK}({\rm{KS}}_r) \subset {\rm{PKS}}_r$.

In Step 1.3, we show $G_{PK}\circ F_{PK}=F_{PK}\circ G_{PK}=\Id$, and this completes Step 1.
\\
\textbf{Step 1.1.} 

Suppose $\pkappa=\{\pkappa_\frn \}_\frn \in {\rm{PKS}}_r.$ Put 
$$\kappa_\frn=F_{PK}(\pkappa)_\frn=\sum_{\frd \subset \frn}(-1)^{\nu(\frn/\frd)}\pi_\frn(\pkappa_\frd)\prod_{\frq \in \frn/\frd}\pi_{\frn/\frq}(P_\frq).$$
We show that $\kappa=\{ \kappa_\frn \}_\frn =F_{PK}(\pkappa) \in {\rm{KS}}_r$. We see that $\kappa$ satisfies the axioms (K1)-(K4).
\\
(K1) Suppose $\frq' \in \Sigma \setminus \frn$. We have 
$$v_{\frq'}(\kappa_\frn)=\sum_{\frd \subset \frn}(-1)^{\nu(\frn/\frd)}\pi_\frn(v_{\frq'}(\pkappa_\frd))\prod_{\frq \in \frn/\frd}\pi_{\frn/\frq}(P_\frq)=0,$$
since $v_{\frq'}(\pkappa_\frd)=0$ for every $\frd \subset \frn$, by (PK1). This shows (K1).

From now on we suppose $\frq' \in \frn$. 
\\
(K2) By (PK2), we have
$$u_{\frq'}(\kappa_\frn)=u_{\frq'}\left( \sum_{\frd \subset \frn}(-1)^{\nu(\frn/\frd)}\pi_\frn(\pkappa_\frd)\prod_{\frq \in \frn/\frd}\pi_{\frn/\frq}(P_\frq) \right)=0.$$
This shows (K2).
\\
(K3) We have
\begin{eqnarray}
v_{\frq'}(\kappa_\frn)&=&\sum_{\frd \subset \frn}(-1)^{\nu(\frn/\frd)}\pi_\frn(v_{\frq'}(\pkappa_\frd))\prod_{\frq \in \frn/\frd}\pi_{\frn/\frq}(P_\frq) \nonumber \\
&=& \sum_{\frd \subset \frn, \frq' \in \frd}(-1)^{\nu(\frn/\frd)}\pi_\frn(v_{\frq'}(\pkappa_\frd))\prod_{\frq \in \frn/\frd}\pi_{\frn/\frq}(P_\frq) \nonumber \\
&=& \varphi_{\frq'}^\frn\left( \sum_{\frd \subset \frn/\frq'}(-1)^{\nu(\frn/\frd\frq')}\pi_\frn(\pkappa_\frd)\prod_{\frq \in \frn/\frd\frq'}\pi_{\frn/\frq}(P_\frq) \right) \nonumber \\
&=& \varphi_{\frq'}^\frn\left( \sum_{\frd \subset \frn/\frq'}(-1)^{\nu(\frn/\frd\frq')}\pi_{\frn/\frq'}(\pkappa_\frd)\prod_{\frq \in \frn/\frd\frq'}\pi_{\frn/\frq\frq'}(P_\frq) \right) \nonumber \\
&=& \varphi_{\frq'}^\frn(\kappa_{\frn/\frq'}) \nonumber \\
&=&\varphi_{\frq'}(\kappa_{\frn/\frq'}) \nonumber,
\end{eqnarray}
where the second equality follows from (PK1), that is, $v_{\frq'}(\pkappa_\frd)=0$ unless $\frq' \in \frd$, and the third from (PK3), 
that is, $v_{\frq'}(\pkappa_\frd)=\varphi_{\frq'}(\pkappa_{\frd/\frq'})$, the fourth from proposition \ref{prop1} (ii), fifth by definition, and the last from (K1).
\\
(K4) By Lemma \ref{lem2}, we have 
$$\pi_{\frn/\frq'}(\kappa_\frn)=\pi_{\frn/\frq'}\left(\sum_{\frd \subset \frn}(-1)^{\nu(\frn/\frd)}\pi_\frn(\pkappa_\frd)\prod_{\frq \in \frn/\frd}\pi_{\frn/\frq}(P_\frq)\right)=0.$$

Hence we have $\kappa \in {\rm{KS}}_r$.
\\
\textbf{Step 1.2.}

We construct the inverse $G_{PK}$ of $F_{PK}$. Suppose $\kappa=\{\kappa_\frn\}_\frn \in {\rm{KS}}_r$ is given. Put 
$$\pkappa_1=\kappa_1,$$
and define $\pkappa_\frn \in \bigwedge^rH\otimes_\cO G(\Sigma)_{\nu(\frn)}$ inductively by
\begin{equation}
\pkappa_\frn=\kappa_\frn+\sum_{\frd \subset \frn, \frd \neq \frn}\pi_\frn(\pkappa_\frd)\left\{ \left( \prod_{\frq \in \frn/\frd}P_\frq|_{\Sigma \setminus \frn} \right)-(-1)^{\nu(\frn/\frd)}\left( \prod_{\frq \in \frn/\frd}\pi_{\frn/\frq}(P_\frq) \right) \right\}. \label{eq2}
\end{equation}
We define $G_{PK}(\kappa)=\{ \pkappa_\frn \}_\frn$. We show first that $\pkappa=\{ \pkappa_\frn \}_\frn=G_{PK}(\kappa) \in {\rm{PKS}}_r$ (in Step 1.3 we show that 
$G_{PK}\circ F_{PK}=F_{PK}\circ G_{PK}=\Id$).
\\
(PK1) We show by induction on $\nu(\frn)$ that $v_{\frq'}(\pkappa_\frn)=0$ for $\frq' \in \Sigma \setminus \frn$. When $\nu(\frn)=0$ i.e. 
$\frn=1$, this is clear by (K1) since $\pkappa_1=\kappa_1$. When $\nu(\frn)>0$, we have for $\frq' \in \Sigma \setminus \frn$ 
\begin{eqnarray}
v_{\frq'}(\pkappa_\frn)&=&v_{\frq'}(\kappa_\frn)  + \sum_{\frd \subset \frn, \frd \neq \frn}\pi_\frn(v_{\frq'}(\pkappa_\frd))\left\{ \left( \prod_{\frq \in \frn/\frd}P_\frq|_{\Sigma \setminus \frn} \right)-(-1)^{\nu(\frn/\frd)}\left( \prod_{\frq \in \frn/\frd}\pi_{\frn/\frq}(P_\frq) \right) \right\} \nonumber \\ 
&=&0, \nonumber
\end{eqnarray}
by (K1) and the inductive hypothesis. This shows (PK1).
\\
(PK2) Applying $\pi_\frn$ to the both sides of (\ref{eq2}), we obtain 
\begin{equation}
\pi_\frn(\pkappa_\frn)=\kappa_\frn-\sum_{\frd \subset \frn, \frd \neq \frn}(-1)^{\nu(\frn/\frd)}\pi_\frn(\pkappa_\frd)\prod_{\frq \in \frn/\frd}\pi_{\frn/\frq}(P_\frq). \label{eq3} 
\end{equation}
Hence by (K2) we have 
$$u_{\frq'}\left( \sum_{\frd \subset \frn}(-1)^{\nu(\frn/\frd)}\pi_\frn(\pkappa_\frd)\prod_{\frq \in \frn/\frd}\pi_{\frn/\frq}(P_\frq) \right)=
u_{\frq'}(\kappa_\frn)=0,$$
for any $\frq' \in \frn$. This shows (PK2).

Next we show (PK5), (PK4), and finally (PK3).
\\
(PK5) By (\ref{eq3}), we have 
\begin{equation}
\kappa_\frn=\sum_{\frd \subset \frn}(-1)^{\nu(\frn/\frd)}\pi_\frn(\pkappa_\frd)\prod_{\frq \in \frn/\frd}\pi_{\frn/\frq}(P_\frq). \label{eq4} 
\end{equation}
Substituting this to (\ref{eq2}), we obtain 
\begin{eqnarray}
&&\pkappa_\frn \nonumber \\
&=&\sum_{\frd \subset \frn}(-1)^{\nu(\frn/\frd)}\pi_\frn(\pkappa_\frd)\prod_{\frq \in \frn/\frd}\pi_{\frn/\frq}(P_\frq)  +\sum_{\frd \subset \frn, \frd \neq \frn}\pi_\frn(\pkappa_\frd)\left\{ \left( \prod_{\frq \in \frn/\frd}P_\frq|_{\Sigma \setminus \frn} \right)-(-1)^{\nu(\frn/\frd)}\left( \prod_{\frq \in \frn/\frd}\pi_{\frn/\frq}(P_\frq) \right) \right\} \nonumber \\
&=&\sum_{\frd \subset \frn}\pi_\frn(\pkappa_\frd)\prod_{\frq \in \frn/\frd}P_\frq|_{\Sigma\setminus\frn}. \nonumber
\end{eqnarray}
This is (PK5).
\\
(PK4) We show by induction on $\nu(\frn)$ that $\pkappa_\frn|_{\Sigma\setminus\frq'}=\pkappa_{\frn/\frq'}|_{\Sigma\setminus\frq'}\cdot P_\frq'$ for any $\frq' \in \frn$. When $\nu(\frn)=1$, say $\frn=\frq'$, we have by (\ref{eq2}) 
$$\pkappa_{\frq'}|_{\Sigma\setminus\frq'}=\kappa_{\frq'}|_{\Sigma \setminus \frq'}+\pi_{\frq'}(\pkappa_1)P_{\frq'}|_{\Sigma\setminus\frq'}
=\pkappa_1\cdot P_{\frq'},$$
so (PK4) holds in this case.

When $\nu(\frn)>1$, take $\frq' \in \frn$. By (\ref{eq2}) and the fact that $\kappa_\frn|_{\Sigma\setminus\frq'}=0$ (this follows from (K4)), we have
\begin{eqnarray}
\pkappa_\frn|_{\Sigma\setminus\frq'}&=&\sum_{\frd \subset \frn, \frd \neq \frn}\pi_{\frn/\frq'}(\pkappa_\frd)\left\{ \left( \prod_{\frq \in \frn/\frd}P_\frq|_{\Sigma \setminus \frn} \right)-(-1)^{\nu(\frn/\frd)}\pi_{\frn/\frq'}\left( \prod_{\frq \in \frn/\frd}\pi_{\frn/\frq}(P_\frq) \right) \right\} \nonumber \\
&=&\sum_{\frd \subset \frn, \frd \neq \frn, \frq' \in \frd}\pi_{\frn/\frq'}(\pkappa_\frd)\left\{ \left( \prod_{\frq \in \frn/\frd}P_\frq|_{\Sigma \setminus \frn} \right) -(-1)^{\nu(\frn/\frd)}\left( \prod_{\frq \in \frn/\frd}\pi_{\frn/\frq\frq'}(P_\frq) \right) \right\} \nonumber \\
&&  + \sum_{\frd \subset \frn/\frq'}\pi_{\frn/\frq'}(\pkappa_\frd)\left\{ \left( \prod_{\frq \in \frn/\frd}P_\frq|_{\Sigma \setminus \frn} \right)-(-1)^{\nu(\frn/\frd)}\pi_{\frn/\frq'}\left( \prod_{\frq \in \frn/\frd}\pi_{\frn/\frq}(P_\frq) \right) \right\} \nonumber \\
&=&\sum_{\frd \subset \frn/\frq', \frd \neq \frn/\frq'}\pi_{\frn/\frq'}(\pkappa_\frd)\pi_{\frn/\frq'}(P_{\frq'})   \left\{ \left( \prod_{\frq \in \frn/\frd\frq'}P_\frq|_{\Sigma \setminus \frn} \right)-(-1)^{\nu(\frn/\frd\frq')}\left( \prod_{\frq \in \frn/\frd\frq'}\pi_{\frn/\frq\frq'}(P_\frq) \right) \right\}
 \nonumber \\
&& + \sum_{\frd \subset \frn/\frq'}\pi_{\frn/\frq'}(\pkappa_\frd)\left\{ \left( \prod_{\frq \in \frn/\frd}P_\frq|_{\Sigma \setminus \frn} \right)-(-1)^{\nu(\frn/\frd)}\pi_{\frn/\frq'}\left( \prod_{\frq \in \frn/\frd}\pi_{\frn/\frq}(P_\frq) \right) \right\} \nonumber \\
&=& \sum_{\frd \subset \frn/\frq'}\pi_{\frn/\frq'}(\pkappa_{\frd})\pi_{\frn/\frq'}(P_{\frq'})\prod_{\frq \in \frn/\frd\frq'}P_{\frq}|_{\Sigma\setminus\frn} + \sum_{\frd \subset \frn/\frq'}\pi_{\frn/\frq'}(\pkappa_\frd)\prod_{\frq \in \frn/\frd}P_\frq|_{\Sigma\setminus\frn} \nonumber \\
&=&\pkappa_{\frn/\frq'}|_{\Sigma\setminus\frq'}(\pi_{\frn/\frq'}(P_{\frq'})+P_{\frq'}|_{\Sigma\setminus\frn}) \nonumber \\
&=&\pkappa_{\frn/\frq'}|_{\Sigma\setminus\frq'}P_{\frq'}, \nonumber
\end{eqnarray}
where the third equality follows by the inductive hypothesis, and the fifth by (PK5).
\\
(PK3) We show by induction on $\nu(\frn)$ that $v_{\frq'}(\pkappa_\frn)=\varphi_{\frq'}(\pkappa_{\frn/\frq'})$ for any $\frq' \in \frn$. When 
$\nu(\frn)=1$, say $\frn=\frq'$, we have 
\begin{eqnarray}
v_{\frq'}(\pkappa_{\frq'})&=&v_{\frq'}(\kappa_{\frq'})+\pi_{\frq'}(v_{\frq'}(\pkappa_1))P_{\frq'}|_{\Sigma\setminus\frq'} \nonumber \\
&=&v_{\frq'}(\kappa_{\frq'}) \nonumber \\
&=&\varphi_{\frq'}(\kappa_1) \nonumber \\
&=&\varphi_{\frq'}(\pkappa_1), \nonumber
\end{eqnarray}
where the first equality follows by (\ref{eq2}), the second by (PK1), the third by (K3), and the last by the definition of $\pkappa_1$. When 
$\nu(\frn)>1$, take $\frq' \in \frn$. Then we have 
\begin{eqnarray}
v_{\frq'}(\pkappa_\frn)&=&v_{\frq'}(\kappa_\frn)+\sum_{\frd \subset \frn, \frd \neq \frn}\pi_\frn(v_{\frq'}(\pkappa_\frd))  \left\{ \left( \prod_{\frq \in \frn/\frd}P_\frq|_{\Sigma \setminus \frn} \right)-(-1)^{\nu(\frn/\frd)}\left( \prod_{\frq \in \frn/\frd}\pi_{\frn/\frq}(P_\frq) \right) \right\} \nonumber \\
&=& \varphi_{\frq'}^\frn( \kappa_{\frn/\frq'})+\sum_{\frd \subset \frn/\frq', \frd \neq \frn/\frq'}\pi_\frn(\pkappa_\frd)  \left\{ \left( \prod_{\frq \in \frn/\frd\frq'}P_\frq|_{\Sigma \setminus \frn} \right)-(-1)^{\nu(\frn/\frd\frq')}\left( \prod_{\frq \in \frn/\frd\frq'}\pi_{\frn/\frq}(P_\frq) \right) \right\} , \nonumber 
\end{eqnarray}
where the first equality follows by (\ref{eq2}), and the second by (K3) and the inductive hypothesis (note that $v_{\frq'}(\pkappa_\frd)=0$ unless $\frq' \in \frd$, by (PK1)). By (\ref{eq4}) and Proposition \ref{prop1} (ii) (note that we have already proved (PK4)), we have 
$$\kappa_{\frn/\frq'}=\sum_{\frd \subset \frn/\frq'}(-1)^{\nu(\frn/\frd\frq')}\pi_\frn(\pkappa_\frd)\prod_{\frq \in \frn/\frd\frq'}\pi_{\frn/\frq}(P_\frq).$$
Substituting this to the above, we have 
\begin{eqnarray}
v_{\frq'}(\pkappa_\frn) &=& \varphi_{\frq'}^\frn\left( \sum_{\frd \subset \frn/\frq'}\pi_\frn(\pkappa_\frd)\prod_{\frq \in \frn/\frd\frq'}P_\frq|_{\Sigma\setminus\frn} \right)
 \nonumber \\
&=& \varphi_{\frq'}^\frn(\pkappa_{\frn/\frq'}) \nonumber \\
&=&\varphi_{\frq'}(\pkappa_{\frn/\frq'}), \nonumber
\end{eqnarray}
where the second equality follows by (PK5) and Proposition \ref{prop1} (i), and the last by (PK1).

Hence $\kappa$ satisfies the axioms (PK1)-(PK5), and we have completed Step 1.2.
\\
\textbf{Step 1.3.}

In this step, we show $G_{PK}\circ F_{PK}=F_{PK}\circ G_{PK}=\Id$.

We first show $G_{PK}\circ F_{PK}=\Id$. Take any $\pkappa=\{ \pkappa_\frn \}_\frn \in {\rm{PKS}}_r$. We show by induction on $\nu(\frn)$ that 
$(G_{PK}\circ F_{PK})(\pkappa)_\frn=\pkappa_\frn$. When $\nu(\frn)=0$, i.e. $\frn=1$, by the definitions of $F_{PK}$ and $G_{PK}$, we have 
$$(G_{PK}\circ F_{PK})(\pkappa)_1=F_{PK}(\pkappa)_1=\pkappa_1.$$
When $\nu(\frn)>0$, we have 
\begin{eqnarray}
&&(G_{PK}\circ F_{PK})(\pkappa)_\frn  \nonumber \\
&=& F_{PK}(\pkappa)_\frn  + \sum_{\frd \subset \frn, \frd \neq \frn}\pi_\frn((G_{PK}\circ F_{PK})(\pkappa)_\frd)\left\{ \left( \prod_{\frq \in \frn/\frd}P_\frq|_{\Sigma \setminus \frn} \right)-(-1)^{\nu(\frn/\frd)}\left( \prod_{\frq \in \frn/\frd}\pi_{\frn/\frq}(P_\frq) \right) \right\} \nonumber \\
&=& \sum_{\frd \subset \frn}(-1)^{\nu(\frn/\frd)}\pi_\frn(\pkappa_\frd)\prod_{\frq \in \frn/\frd}\pi_{\frn/\frq}(P_\frq)  +\sum_{\frd \subset \frn, \frd \neq \frn}\pi_\frn(\pkappa_\frd)\left\{ \left( \prod_{\frq \in \frn/\frd}P_\frq|_{\Sigma \setminus \frn} \right)-(-1)^{\nu(\frn/\frd)}\left( \prod_{\frq \in \frn/\frd}\pi_{\frn/\frq}(P_\frq) \right) \right\} \nonumber \\
&=& \sum_{\frd \subset \frn}\pi_\frn(\pkappa_\frd)\prod_{\frq \in \frn/\frd}P_\frq|_{\Sigma\setminus\frn} \nonumber \\
&=& \pkappa_\frn, \nonumber
\end{eqnarray}
where the first equality follows by the definition of $G_{PK}$ (see (\ref{eq2})), the second by the definition of $F_{PK}$ (see Definition 
\ref{defmor}) and the inductive hypothesis, and the last by (PK5). 

Next we show $F_{PK}\circ G_{PK}=\Id$. Take any $\kappa=\{ \kappa_\frn\}_\frn \in {\rm{KS}}_r$. By (\ref{eq4}), we have 
$$\kappa_\frn=\sum_{\frd \subset \frn}(-1)^{\nu(\frn/\frd)}\pi_\frn(G_{PK}(\kappa)_\frd)\prod_{\frq \in \frn/\frd}\pi_{\frn/\frq}(P_\frq),$$
but the right hand side is by definition equal to $F_{PK}(G_{PK}(\kappa))_\frn$. We have completed Step 1.3.
\\
\textbf{Step 2.}

We show that $F_{TD}$ induces an isomorphism ${\rm{TKS}}_r \simeq {\rm{DKS}}_r$. Step 2 is divided into 3 steps, as in Step 1. 

In Step 2.1, we show $F_{TD}({\rm{TKS}}_r) \subset {\rm{DKS}}_r.$

In Step 2.2, we construct the inverse $G_{TD}$ of $F_{TD}$, and show $G_{TD}({\rm{DKS}}_r)\subset {\rm{TKS}}_r$.

In Step 2.3, we show $G_{TD}\circ F_{TD}=F_{TD}\circ G_{TD}=\Id.$
\\
\textbf{Step 2.1.}

Take $\theta=\{\theta_\frn\}_\frn \in {\rm{TKS}}_r$. We show that $F_{TD}(\theta) \in {\rm{DKS}}_r$. Put 
$$\kappa'_\frn=F_{TD}(\theta)_\frn=\sum_{\frd \subset \frn}(-1)^{\nu(\frn/\frd)}\theta_{\frd}\prod_{\frq \in \frn/\frd}\pi_{\frd}(P_\frq).$$
Note that by (TK4) we have 
$$\theta_{\frd}\prod_{\frq \in \frn/\frd}\pi_{\frd}(P_\frq)=\pi_\frd(\theta_\frn),$$
so we have 
\begin{equation}
\kappa'_\frn=s_\frn(\theta_\frn) \label{eqs}
\end{equation}
(see Definition \ref{defs} for the definition of $s_\frn$). We see that $\{\kappa'_\frn\}_\frn$ satisfies the axioms (DK1)-(DK4).
\\
(DK1) For any $\frq \in \Sigma\setminus\frn$, we have by (TK1)
$$v_\frq(\kappa'_\frn)=\sum_{\frd \subset \frn}(-1)^{\nu(\frn/\frd)}\pi_\frd(v_\frq(\theta_\frn))=0.$$
This is (DK1).
\\
(DK2) It is sufficient to show that 
\begin{equation}
\sum_{\frd \subset \frn}\theta_\frd\cD_{\frn,\frn/\frd}=\sum_{\frd \subset \frn}\kappa'_\frd \cD_{\frn/\frd}. \label{eq8}  
\end{equation}
(From this, (DK2) follows from (TK2)). Take $\frq \in \frn$. We have 
\begin{eqnarray}
\pi_{\frn/\frq}\left( \sum_{\frd \subset \frn}\theta_\frd\cD_{\frn,\frn/\frd} \right)&=& \pi_{\frn/\frq}\left( \sum_{\frd \subset \frn/\frq}\theta_{\frd\frq}\cD_{\frn,\frn/\frd\frq} +\sum_{\frd \subset \frn/\frq}\theta_\frd\cD_{\frn,\frn/\frd} \right) \nonumber \\
&=& \sum_{\frd \subset \frn/\frq}\theta_\frd\pi_\frd(P_\frq)\cD_{\frn/\frq,\frn/\frd\frq}-\sum_{\frd \subset \frn/\frq}\theta_\frd\cD_{\frn/\frq,\frn/\frd\frq}\pi_\frd(P_\frq) \nonumber \\
&=& 0, \nonumber
\end{eqnarray}
where the second equality follows by Proposition \ref{propd} (i), (ii) and (TK4). So we have by the definition of $s_\frn$
$$s_\frn\left(\sum_{\frd \subset \frn}\theta_\frd\cD_{\frn,\frn/\frd}\right)=\sum_{\frd \subset \frn}\theta_\frd\cD_{\frn,\frn/\frd}.$$
On the other hand, by Lemma \ref{lem1} (ii), Proposition \ref{propd} (iii), and (\ref{eqs}), we have 
$$s_\frn\left(\sum_{\frd \subset \frn}\theta_\frd\cD_{\frn,\frn/\frd}\right)=\sum_{\frd \subset \frn}\kappa'_\frd \cD_{\frn/\frd}.$$
Hence we have $\sum_{\frd \subset \frn}\theta_\frd\cD_{\frn,\frn/\frd}=\sum_{\frd \subset \frn}\kappa'_\frd \cD_{\frn/\frd}$.
\\
(DK3) Since $\kappa'_\frn=s_\frn(\theta_\frn)$, (DK3) follows from (TK3). 
\\
(DK4) Again since $\kappa'_\frn=s_\frn(\theta_\frn)$, (DK4) follows from Lemma \ref{lem1} (i). 

Hence we have completed Step 2.1.
\\
\textbf{Step 2.2.}

We construct the inverse $G_{TD}$ of $F_{TD}$. Suppose $\kappa'=\{\kappa'_\frn\}_\frn \in {\rm{DKS}}_r$. Put 
$$\theta_1=\kappa'_1,$$
and we define $\theta_\frn$ inductively by 
\begin{equation}
\theta_\frn=\kappa'_\frn-\sum_{\frd \subset \frn, \frd \neq \frn}(-1)^{\nu(\frn/\frd)}\theta_\frd\prod_{\frq \in \frn/\frd}\pi_{\frd}(P_\frq). \label{eq5}  
\end{equation}
Define $G_{TD}(\kappa')=\{\theta_\frn\}_\frn$, and we show that $G_{TD}(\kappa') \in {\rm{TKS}}_r$. 
\\
(TK1) follows from (DK1) by induction on $\nu(\frn)$.
\\
(TK4) We show by induction on $\nu(\frn)$. When $\nu(\frn)=1$, say $\frn=\frq'$, we have 
$$\pi_1(\theta_{\frq'})=0=\theta_1 \pi_1(P_{\frq'})$$
(note that $\pi_1(G(\frq')_1)=0$). When $\nu(\frn)>1$, for any $\frq' \in \frn$ we have by (\ref{eq5})
\begin{eqnarray}
\pi_{\frn/\frq'}(\theta_\frn)&=&\pi_{\frn/\frq'}(\kappa'_\frn) -\sum_{\frd \subset \frn/\frq'}(-1)^{\nu(\frn/\frd)}\theta_\frd\prod_{\frq \in \frn/\frd}\pi_\frd(P_\frq) -\sum_{\frd \subset \frn/\frq', \frd \neq \frn/\frq'}(-1)^{\nu(\frn/\frd\frq')}\pi_\frd(\theta_{\frd\frq'})\prod_{\frq \in \frn/\frd\frq'}\pi_\frd(P_\frq) \nonumber \\
&=&-\sum_{\frd \subset \frn/\frq'}(-1)^{\nu(\frn/\frd)}\theta_\frd\prod_{\frq \in \frn/\frd}\pi_\frd(P_\frq) -\sum_{\frd \subset \frn/\frq', \frd \neq \frn/\frq'}(-1)^{\nu(\frn/\frd\frq')}\theta_\frd\pi_\frd(P_{\frq'})\prod_{\frq \in \frn/\frd\frq'}\pi_\frd(P_\frq) \nonumber \\
&=&\theta_{\frn/\frq'}\pi_{\frn/\frq'}(P_{\frq'}), \nonumber
\end{eqnarray}
where the second equality follows from (DK4) and the inductive hypothesis. This shows (TK4).
\\
(TK2) By (\ref{eq5}) and (TK4), we have 
$$\theta_\frn=\kappa'_\frn-\sum_{\frd \subset \frn, \frd \neq \frn}(-1)^{\nu(\frn/\frd)}\pi_\frd(\theta_\frn).$$
Hence, 
\begin{equation}
\kappa'_\frn=\frs_n(\theta_\frn). \label{eq6}  
\end{equation}
Using (\ref{eq6}) and (TK4), we repeat the argument in the proof of (DK2) in Step 2.1 to show $\sum_{\frd \subset \frn}\theta_\frd\cD_{\frn,\frn/\frd}=\sum_{\frd \subset \frn}\kappa'_\frd \cD_{\frn/\frd}$. Hence (TK2) follows from (DK2). 
\\
(TK3) By (\ref{eq6}), (TK3) follows from (DK3). 

We have completed Step 2.2.
\\
\textbf{Step 2.3.}

To show that $F_{TD}$ induces isomorphism from ${\rm{TKS}}_r$ to ${\rm{DKS}}_r$, since we already know by Proposition \ref{propinj} that $F_{TD}$ is injective, it suffices to show $F_{TD} \circ G_{TD}=\Id$. Suppose $\kappa'=\{ \kappa'_\frn \}_\frn \in {\rm{DKS}}_r$. By (\ref{eq6}) we have 
$$\kappa'_\frn=\sum_{\frd \subset \frn}(-1)^{\nu(\frn/\frd)}\pi_\frd(G_{TD}(\kappa')_\frn)=s_\frn(G_{TD}(\kappa')_\frn).$$
By (\ref{eqs}) we have 
$$F_{TD}(G_{TD}(\kappa'))_\frn=s_\frn(G_{TD}(\kappa')_\frn),$$
which completes Step 2.3.
\\
\textbf{Step 3.}

Since the bijectivity of $F_{TK}$ is shown similarly as in Step 2 (or in the proof of \cite[Proposition 6.5]{MR2}), we omit the proof. 
\\
\textbf{Step 4.}
 
We show $F_{DK}\circ F_{TD}=F_{TK}$. Take $\theta=\{ \theta_\frn\}_\frn \in {\rm{TKS}}_r$. We have to show  
$$\sum_{\frd \subset \frn}F_{TD}(\theta)_\frd\cD_{\frn/\frd}=\sum_{\frd \subset \frn}\theta_\frd \cD_{\frn,\frn/\frd}.$$
But this is (\ref{eq8}), which has been already shown. Hence $F_{DK}\circ F_{TD}=F_{TK}$.
\\
\textbf{Step 5.}

Our final task is to prove $F_{TK}\circ F_{PT}=F_{PK}$. Take $\pkappa=\{\pkappa_\frn\}_\frn \in {\rm{PKS}}_r$. We have to prove 
\begin{equation}
\sum_{\frd\subset\frn}\pi_\frd(\pkappa_\frd)\cD_{\frn,\frn/\frd}=\sum_{\frd \subset \frn}(-1)^{\nu(\frn/\frd)}\pi_\frn(\pkappa_\frd)\prod_{\frq \in \frn/\frd}\pi_{\frn/\frq}(P_\frq). \label{eq9}
\end{equation}
By (PK5), we have for $\frd \subset \frn$
$$\pi_\frn(\pkappa_\frd)=\sum_{\frc \subset \frd}\pi_{\frd}(\pkappa_\frc)\prod_{\frq \in \frd/\frc}\pi_{\frn/\frd}(P_\frq).$$
Using this relation repeatedly, we arrange the right hand side of (\ref{eq9}), and sum up the ``coefficients" of each $\pi_\frd(\pkappa_\frd)$ to obtain 
\begin{eqnarray}
&&\sum_{\frd \subset \frn}(-1)^{\nu(\frn/\frd)}\pi_\frn(\pkappa_\frd)\prod_{\frq \in \frn/\frd}\pi_{\frn/\frq}(P_\frq) \nonumber \\
&=&\sum_{\frd\subset\frn}\left( \sum_{(\frc_1,\ldots,\frc_k)\in \Delta(\frn/\frd)}(-1)^{\nu(\frc_k)}\prod_{\frq\in \frc_k}\pi_{\frn/\frq}(P_\frq)  \prod_{\frq\in \frc_{k-1}}\pi_{\frc_k}(P_\frq)\prod_{\frq \in \frc_{k-2}}\pi_{\frc_{k-1}}(P_\frq)\cdots\prod_{\frq\in \frc_1}\pi_{\frc_2}(P_\frq) \right)\pi_\frd(\pkappa_\frd), \nonumber 
\end{eqnarray}
where 
$$\Delta(\frn/\frd)=\{ (\frc_1,\ldots,\frc_k) \ | \ \emptyset\neq \frc_i \subset \frn/\frd, \  \frn/\frd=\coprod_{i=1}^k\frc_i, \ k\in \bbZ_{\geq1} \}.$$
Hence it is sufficient to show 
$$\cD_{\frn,\frn/\frd}=\sum_{(\frc_1,\ldots,\frc_k)\in \Delta(\frn/\frd)}(-1)^{\nu(\frc_k)}\prod_{\frq\in \frc_k}\pi_{\frn/\frq}(P_\frq)\prod_{\frq\in \frc_{k-1}}\pi_{\frc_k}(P_\frq)\prod_{\frq \in \frc_{k-2}}\pi_{\frc_{k-1}}(P_\frq)\cdots\prod_{\frq\in \frc_1}\pi_{\frc_2}(P_\frq).$$
This is reduced to the following

\begin{lemma}
Suppose $A=(a_{ij})$ is a $\nu\times\nu$-matrix with entries in a commutative ring, then we have 
\begin{eqnarray}
&&(-1)^\nu|A| \nonumber \\
&=&\sum_{(C_1,\ldots,C_k)\in \Delta(\nu)}(-1)^{|C_k|}\prod_{i\in C_k}\left( \sum_{j=1}^\nu a_{ij} \right)  \prod_{i\in C_{k-1}}\left(  \sum_{j\in C_k}a_{ij}\right)\prod_{i\in C_{k-2}}\left( \sum_{j\in C_{k-1}}a_{ij} \right)\cdots\prod_{i\in C_1}\left( \sum_{j\in C_2}a_{ij} \right), \nonumber
\end{eqnarray}
where $\Delta(\nu)=\{ (C_1,\ldots,C_k) \ | \ \emptyset\neq C_i \subset \{ 1, \ldots , \nu\}, \ \{1,\ldots,\nu\}=\coprod_{i=1}^kC_i, \ k\in\bbZ_{\geq1} \}.$
\end{lemma}

\begin{proof}
Fix a map $\tau : \{ 1,\ldots,\nu\}\to \{ 1,\ldots,\nu \}$. We see that the coefficient of $\prod_{i=1}^\nu a_{i,\tau(i)}$ of the left (resp. right) hand side of the equality in the lemma is 
$$\sum_{\sigma \in \mathfrak{S}_\nu}{\rm sgn}(\sigma) \prod_{i=1}^\nu (-\delta_{\tau(i), \sigma(i)})$$
$$\text{(resp. } \sum_{(C_1,\ldots,C_k)\in \Delta(\tau)}(-1)^{|C_k|}),$$
where $\delta_{\tau(i),\sigma(i)}$ denotes Kronecker's delta, and 
$$\Delta(\tau)=\{ (C_1,\ldots,C_k)\in \Delta(\nu) \ | \ \tau(i) \in C_{j+1} \text{ for all } 1\leq j \leq k-1 \text{ and } i\in C_j \}.$$
So it is sufficient to show that
$$\sum_{\sigma \in \mathfrak{S}_\nu}{\rm sgn}(\sigma) \prod_{i=1}^\nu (-\delta_{\tau(i), \sigma(i)})= \sum_{(C_1,\ldots,C_k)\in \Delta(\tau)}(-1)^{|C_k|}.$$
For every map $\mu : \{ 1,\ldots,\nu\} \to \{ 1, \ldots,\nu\}$ we set 
$${\rm Fix}(\mu)=\{ i\in \{1,\ldots,\nu\} \ | \ \mu(i)=i \}.$$
We compute
\begin{eqnarray}
&&\sum_{\sigma \in \mathfrak{S}_\nu}{\rm sgn}(\sigma) \prod_{i=1}^\nu (-\delta_{\tau(i), \sigma(i)})\nonumber \\
&=& \sum_{\sigma \in \mathfrak{S}_\nu}{\rm sgn}(\sigma)\prod_{i\in {\rm Fix}(\tau)}(-\delta_{i,\sigma(i)})\prod_{i\notin {\rm Fix}(\tau)}(-\delta_{\tau(i),\sigma(i)}) \nonumber \\
&=& \sum_{\sigma \in \mathfrak{S}_\nu , {\rm Fix}(\tau)\subset {\rm Fix(\sigma)}}{\rm sgn}(\sigma)(-1)^{|{\rm Fix}(\tau)|}\prod_{i\notin {\rm Fix}(\tau)}(-\delta_{\tau(i),\sigma(i)}) \nonumber \\
&=&\sum_{\sigma \in \mathfrak{S}_\nu , {\rm Fix}(\tau)\subset {\rm Fix(\sigma)}}{\rm sgn}(\sigma)(-1)^{|{\rm Fix}(\tau)|}(1-1)^{|{\rm Fix}(\sigma)\setminus{\rm Fix}(\tau)|}  \prod_{i\notin {\rm Fix}(\sigma)}(-\delta_{\tau(i),\sigma(i)}) \nonumber \\
&=&\sum_{\sigma \in \mathfrak{S}_\nu , {\rm Fix}(\tau)\subset {\rm Fix(\sigma)}}{\rm sgn}(\sigma)\sum_{D\subset {\rm Fix}(\sigma)\setminus{\rm Fix}(\tau)}(-1)^{|D|+|{\rm Fix}(\tau)|}\prod_{i\notin {\rm Fix}(\sigma)}(-\delta_{\tau(i),\sigma(i)}) \nonumber \\
&=&\sum_{{\rm Fix}(\tau)\subset C \subset \{ 1,\ldots,\nu\}}(-1)^{|C|}\sum_{\sigma\in \mathfrak{S}_\nu,C\subset {\rm Fix}(\sigma)}{\rm sgn}(\sigma)\prod_{i\notin {\rm Fix}(\sigma)}(-\delta_{\tau(i),\sigma(i)}). \nonumber
\end{eqnarray}
Note that 
\begin{eqnarray}
&&\sum_{(C_1,\ldots,C_k)\in \Delta(\tau)}(-1)^{|C_k|}\nonumber \\
&=&\sum_{C \subset \{ 1,\ldots,\nu\}}(-1)^{|C|}|\{ (C_1,\ldots,C_k)\in \Delta(\tau) \ | \ C_k=C \}|\nonumber \\
&=&\sum_{{\rm Fix}(\tau)\subset C \subset \{ 1,\ldots,\nu\}}(-1)^{|C|}|\{ (C_1,\ldots,C_k)\in \Delta(\tau) \ | \ C_k=C \}|.\nonumber 
\end{eqnarray}
Hence, it is sufficient to show for each set $C$ with ${\rm Fix}(\tau)\subset C \subset \{1,\ldots,\nu\}$ that
$$\sum_{\sigma \in \mathfrak{S}_\nu,C\subset {\rm Fix}(\sigma)}{\rm sgn}(\sigma)\prod_{i\notin {\rm Fix}(\sigma)}(-\delta_{\tau(i),\sigma(i)})=|\{ (C_1,\ldots,C_k)\in \Delta(\tau) \ | \ C_k=C \}|.$$
Note that the right hand side is equal to $1$ or $0$. Suppose first that the right hand side is equal to $1$. Then we see that $\prod_{i\notin {\rm Fix}(\sigma)}(-\delta_{\tau(i),\sigma(i)})=0$ unless $\sigma=\Id$. Indeed, suppose $\sigma \neq \Id$ and let $(C_1,\ldots,C_k)$ be the unique element of $\{ (C_1,\ldots,C_k)\in \Delta(\tau) \ | \ C_k=C \}$. Note that in this case we must have $k\geq 2$, since $C\subset {\rm Fix}(\sigma)$. We see that there exists an integer $j$ with $1\leq j \leq k-1$ such that $C_j \not\subset {\rm Fix}(\sigma)$ and $C_{j+1}\subset {\rm Fix}(\sigma)$. This shows that there exists $i\in C_j$ such that $i\notin {\rm Fix}(\sigma)$ and $\tau(i)\neq\sigma(i)$ (since $\sigma$ is injective). Hence we have shown that $\prod_{i\notin {\rm Fix}(\sigma)}(-\delta_{\tau(i),\sigma(i)})=0$ unless $\sigma=\Id$. Therefore we have 
\begin{eqnarray}
\sum_{\sigma \in \mathfrak{S}_\nu,C\subset {\rm Fix}(\sigma)}{\rm sgn}(\sigma)\prod_{i\notin {\rm Fix}(\sigma)}(-\delta_{\tau(i),\sigma(i)})&=&{\rm sgn}(\Id) \nonumber \\
&=&1 \nonumber \\
&=&|\{ (C_1,\ldots,C_k)\in \Delta(\tau) \ | \ C_k=C \}|.\nonumber
\end{eqnarray}
Next, suppose that $|\{ (C_1,\ldots,C_k)\in \Delta(\tau) \ | \ C_k=C \}|=0$. In this case we must have $C \neq \{ 1,\ldots,\nu\}$, and we see that there exist $j \in\{ 1,\ldots,\nu\}\setminus C$ and a positive integer $m$ such that $\tau^{m+1}(j)=j$, that $j,\tau(j),\ldots,\tau^{m}(j)$ are different each other and not contained in $C$. We set $\mu=(j \ \tau(j) \  \cdots \  \tau^{m}(j))\in \mathfrak{S}_\nu$. If we put 
$$\mathfrak{S}_\nu(\tau,C)=\{ \sigma \in\mathfrak{S}_\nu \ | \ C \subset {\rm Fix}(\sigma), \ \tau(i)=\sigma(i) \text{ for all } i\notin {\rm Fix}(\sigma) \},$$
then we have 
$$\sum_{\sigma \in \mathfrak{S}_\nu,C\subset {\rm Fix}(\sigma)}{\rm sgn}(\sigma)\prod_{i\notin {\rm Fix}(\sigma)}(-\delta_{\tau(i),\sigma(i)})=\sum_{\sigma \in \mathfrak{S}_\nu(\tau,C)}{\rm sgn}(\sigma)(-1)^{\nu-|{\rm Fix}(\sigma)|}.$$
It is easy to see that 
$$\{\sigma \in \mathfrak{S}_\nu(\tau,C) \ | \ \sigma(j)\neq j\}=\mu \{ \sigma\in\mathfrak{S}_\nu(\tau,C) \ | \ \sigma(j)=j \},$$
and therefore we have
$$\mathfrak{S}_\nu(\tau,C)=\mu \{ \sigma\in\mathfrak{S}_\nu(\tau,C) \ | \ \sigma(j)=j \}\sqcup \{ \sigma\in\mathfrak{S}_\nu(\tau,C) \ | \ \sigma(j)=j \}.$$
So we have 
\begin{eqnarray}
&&\sum_{\sigma \in \mathfrak{S}_\nu(\tau,C)}{\rm sgn}(\sigma)(-1)^{\nu-|{\rm Fix}(\sigma)|}\nonumber \\
&=&\sum_{\sigma \in \mathfrak{S}_\nu(\tau,C), \sigma(j)=j}{\rm sgn}(\mu\sigma)(-1)^{\nu-|{\rm Fix}(\mu\sigma)|}+\sum_{\sigma \in \mathfrak{S}_\nu(\tau,C), \sigma(j)=j}{\rm sgn}(\sigma)(-1)^{\nu-|{\rm Fix}(\sigma)|}\nonumber \\
&=&({\rm sgn}(\mu)(-1)^{m+1}+1)\sum_{\sigma \in\mathfrak{S}_\nu(\tau,C), \sigma(j)=j}{\rm sgn}(\sigma)(-1)^{\nu-|{\rm Fix}(\sigma)|} \nonumber \\
&=& ((-1)^m(-1)^{m+1}+1)\sum_{\sigma\in \mathfrak{S}_\nu(\tau,C), \sigma(j)=j}{\rm sgn}(\sigma)(-1)^{\nu-|{\rm Fix}(\sigma)|} \nonumber \\
&=&0. \nonumber
\end{eqnarray}
Hence we have 
$$\sum_{\sigma \in \mathfrak{S}_\nu,C\subset {\rm Fix}(\sigma)}{\rm sgn}(\sigma)\prod_{i\notin {\rm Fix}(\sigma)}(-\delta_{\tau(i),\sigma(i)})=0=|\{ (C_1,\ldots,C_k)\in \Delta(\tau) \ | \ C_k=C \}|.$$
This completes the proof.
\end{proof}

Hence we have completed all the steps, therefore the proof of Theorem \ref{thmKS}.

\end{proof}

\section{Regulator Kolyvagin systems} \label{secreg}  

In this section, we construct Kolyvagin systems by ``regulators". We construct an $\cO$-module ${\rm{US}}_r$, which we call ``unit systems" (see Definition \ref{defu} below), and maps from unit systems to Kolyvagin systems (see Theorem \ref{thmreg}). The idea of our method in this section is due to \cite[Appendix B]{MR1}. We keep the notations in \S \ref{secKoly}.

\begin{definition} \label{defsel}  
For $\frn \in \cN$, we define ``$\frn$-modified Selmer group" by 
$$S^\frn=\{ a\in H \ | \ v_\frq(a)=0 {\rm{ \ for \ every \ }} \frq\in\Sigma\setminus\frn \}.$$
\end{definition}

\begin{remark}
In the setting of Example \ref{ex1}, we have
$$S^n=H^1_{\cF^n}(\bbQ,A).$$
\end{remark}

\begin{definition} \label{defu}
Define a partially ordered set 
$$\sI=\{ (s,\cU) \ | \ s=(\frq_1, \frq_2,\ldots) \mbox{ : a sequence of all the elements in } \Sigma, \ \cU\subset \cN \ {\rm{satisfying}} \ (*) \},$$
where 
$$(*) \ \cU=\{\frn_1,\frn_2,\ldots \}, \frn_1\subset\frn_2\subset\cdots\subset\bigcup_{i=1}^\infty\frn_i=\Sigma, \text{ and  } \frn_i=\{\frq_1,\ldots,\frq_{\nu(\frn_i)}\} \text{ for any }  i\geq1,$$
and we define the order on $\sI$ by 
$$(s,\cU)\leq(s',\cU') \text{ if and only if $s=s'$ and $\cU' \subset \cU$}.$$
We define the module of unit systems of rank $r$ ${\rm{US}}_r$ by  
$${\rm{US}}_r=\mathop{\dlim}_{(s,\cU)\in \sI}\mathop{\ilim}_{\frn \in \cU}\bigwedge^{\nu(\frn)+r}S^{\frn},$$
where the morphisms of the inverse limit are defined by 
$$(-v_{\frq_{\nu(\frn_i)+1}})\wedge\cdots\wedge(-v_{\frq_{\nu(\frn_{i+1})}}) : \bigwedge^{\nu(\frn_{i+1})+r}H \longrightarrow 
\bigwedge^{\nu(\frn_i)+r}H,$$
and that of the direct limit by the natural projection maps. 
\end{definition}

\begin{remark}
The assumption that $\Sigma$ is countable is used here.
\end{remark}

\begin{definition} \label{defreg}
Suppose $(s,\cU)\in\sI$, say $s=(\frq_1,\frq_2,\ldots)$, $\cU$ is as $(*)$ above, and $\varepsilon=\{ \varepsilon_{\frn} \}_\frn \in \mathop{\ilim}_{\frn \in \cU}\bigwedge^{\nu(\frn)+r}S^{\frn} $. For $\frn\in\cN$, take $\frn_i\in\cU$ so that $\frn \subset \frn_i$ (this is possible since $\cU$ consists of an increasing sequence of elements in $\cN$ which covers $\Sigma$). Define regulators $R_P(\varepsilon)_\frn$, $R_T(\varepsilon)_\frn$, and $R_K(\varepsilon)_\frn$ by
$$R_P(\varepsilon)_\frn =(\psi_{P,1}^{(\frn)}\wedge \cdots\wedge\psi_{P,\nu(\frn_i)}^{(\frn)})(\varepsilon_{\frn_i}),$$
$$R_T(\varepsilon)_\frn =(\psi_{T,1}^{(\frn)}\wedge \cdots\wedge\psi_{T,\nu(\frn_i)}^{(\frn)})(\varepsilon_{\frn_i}), $$
$$R_K(\varepsilon)_\frn =(\psi_{K,1}^{(\frn)}\wedge \cdots\wedge\psi_{K,\nu(\frn_i)}^{(\frn)})(\varepsilon_{\frn_i}),$$
where 
$$\psi_{P,j}^{(\frn)} \ (\mbox{resp. }\psi_{T,j}^{(\frn)}, \mbox{ resp. }\psi_{K,j}^{(\frn)} )=
\begin{cases}
\varphi_{\frq_j} \ (\mbox{resp.} \ \varphi_{\frq_j}^\frn, \ \mbox{resp.} \ \varphi_{\frq_j}^{\frq_j}) & \mbox{if $\frq_j\in\frn$}, \\
-v_{\frq_j} & \mbox{if $\frq_j \in \frn_i/\frn$},
\end{cases}
 $$
(for the definition of $\varphi_\frq$, see Definition \ref{defrec}). One sees by definition that $R_P(\varepsilon)_\frn$, $R_T(\varepsilon)_\frn$, and $R_K(\varepsilon)_\frn$ do not depend on the choice of $\frn_i$. Indeed, if we take another $\frn_{i'} \in \cU$, say $\frn\subset\frn_i\subset\frn_{i'}$, then we have  
\begin{eqnarray}
(\psi_1^{(\frn)}\wedge\cdots\wedge\psi_{\nu(\frn_{i'})}^{(\frn)})(\varepsilon_{\frn_{i'}})&=&
(\psi_1^{(\frn)}\wedge\cdots\wedge\psi_{\nu(\frn_{i})}^{(\frn)})((-v_{\frq_{\nu(\frn_i)+1}})\wedge\cdots\wedge(-v_{\frq_{\nu(\frn_{i'})}})(\varepsilon_{\frn_{i'}})) \nonumber \\
&=&(\psi_1^{(\frn)}\wedge\cdots\wedge\psi_{\nu(\frn_{i})}^{(\frn)})(\varepsilon_{\frn_{i}}), \nonumber
\end{eqnarray}
where $\psi_j^{(\frn)}$ denotes any of $\psi_{P,j}^{(\frn)}$, $\psi_{T,j}^{(\frn)}$, and $\psi_{K,j}^{(\frn)}$. $R_P$ (resp. $R_T$ and $R_K$) define(s) a homomorophism from ${\rm{US}}_r$ to $\prod_{\frn \in \cN}\bigwedge^rH\otimes_\cO G(\Sigma)_{\nu(\frn)}$ (resp. $\prod_{\frn \in \cN}\bigwedge^rH\otimes_\cO G(\frn)_{\nu(\frn)}$).
\end{definition}

\begin{remark}
The idea of defining the unit systems and the regulators above is due to \cite[Appendix B]{MR1}.
\end{remark}

\begin{theorem} \label{thmreg}  
We have the following commutative diagram:
\[\xymatrix{
{\rm{US}}_r \ar@(d,l)[ddr]_{R_K} \ar[dr]_{R_P} \ar@(r,ul)[drr]^{R_T} & & \\  
 & {\rm{PKS}}_r \ar[d]^{F_{PK}} \ar[r]_{F_{PT}} & {\rm{TKS}}_r \ar[dl]^{F_{TK}} \\
 & {\rm{KS}}_r & \\
}\]

\end{theorem}

\begin{proof}
We first show the commutativity of the diagram, and then prove the image of the map $R_P$ is in ${\rm{PKS}}_r$. This completes the proof of the theorem, since by Theorem \ref{thmKS} we know that $F_{PT}({\rm{PKS}}_r) ={\rm{TKS}}_r$ and $F_{PK}({\rm{PKS}}_r)={\rm{KS}}_r$.

Take $\varep=\{\varep_\frn\}_\frn \in  \mathop{\ilim}_{\frn \in \cU}\bigwedge^{\nu(\frn)+r}S^{\frn}$. To prove the commutativity of the diagram, 
we have to show $R_T(\varep)_\frn=F_{PT}(R_P(\varep))_\frn$ and $R_K(\varep)_\frn=F_{PK}(R_P(\varep))_\frn$ for any $\frn \in \cN$ (note that 
$F_{TK}\circ F_{PT}=F_{PK}$ was already proved in Theorem \ref{thmKS}). Note that by definition $F_{PT}(R_P(\varep))_\frn=\pi_\frn(R_P(\varep)_\frn)$ (see Definition \ref{defmor}), and that $\varphi_\frq^\frn=\pi_\frn\circ\varphi_\frq$, so we have $R_T(\varep)_\frn=F_{PT}(R_P(\varep))_\frn$ by the definitions of $R_T$ and $R_P$. Next, to see $R_K(\varep)_\frn=F_{PK}(R_P(\varep))_\frn$, note that by definition 
$$F_{PK}(R_P(\varep))_\frn=\sum_{\frd \subset \frn}(-1)^{\nu(\frn/\frd)}\pi_\frn(R_P(\varep)_\frd)\prod_{\frq \in \frn/\frd}\pi_{\frn/\frq}(P_\frq),$$
and that
$$\varphi_\frq^\frq=\varphi_\frq^\frn-\varphi_\frq^{\frn/\frq}=\pi_\frn\circ\varphi_\frq-(-v_\frq\cdot\pi_{\frn/\frq}(P_\frq))$$
holds for $\frq \in \frn$, then we see again by definition $R_K(\varep)_\frn=F_{PK}(R_P(\varep))_\frn$ holds (substitute $$\varphi_\frq^\frq=\pi_\frn\circ\varphi_\frq-(-v_\frq\cdot\pi_{\frn/\frq}(P_\frq))$$ to the 
definition of $R_K$, and expand it, then we obtain 
$\sum_{\frd \subset \frn}(-1)^{\nu(\frn/\frd)}\pi_\frn(R_P(\varep)_\frd)\prod_{\frq \in \frn/\frd}\pi_{\frn/\frq}(P_\frq)$).

We prove $R_P(\varep) \in {\rm{PKS}}_r$. Take $\frn_i \in \cU$ so that $\frn\subset \frn_i$. We show that $R_P(\varep)_\frn$ satisfies axioms (PK1)-(PK5).
\\
(PK1) If $\frq \in \Sigma\setminus\frn_i$, we have 
$$v_\frq(R_P(\varep)_\frn)=(v_\frq\wedge\psi_{P,1}^{(\frn)}\wedge \cdots\wedge\psi_{P,\nu(\frn_i)}^{(\frn)})(\varepsilon_{\frn_i})=0,$$ 
since any element $a \in S^{\frn_i}$ satisfies $v_\frq(a)=0$ by definition (see Definition \ref{defsel}). If $\frq \in \frn_i\setminus\frn$, say $\frq=\frq_j$, $1\leq j \leq \nu(\frn_i)$ (recall $\frn_i=\{ \frq_1, \ldots, \frq_{\nu(\frn_i)} \},$ see $(*)$ in Definition \ref{defu}), we have 
\begin{eqnarray}
v_{\frq_j}(R_P(\varep)_\frn)&=&(v_{\frq_j}\wedge\psi_{P,1}^{(\frn)}\wedge \cdots\wedge\psi_{P,\nu(\frn_i)}^{(\frn)})(\varepsilon_{\frn_i}) \nonumber \\
&=& (v_{\frq_j}\wedge\cdots\wedge(-v_{\frq_j})\wedge \cdots)(\varep_{\frn_i}) \nonumber \\
&=&0, \nonumber
\end{eqnarray}
since $(v_{\frq_j}\wedge\cdots\wedge(-v_{\frq_j})\wedge \cdots)=0$. Hence we have $v_\frq(R_P(\varep)_\frn)=0$ for any $\frq \in \Sigma\setminus\frn$. 
\\
(PK2) Take any $\frq \in \frn$. We prove $u_{\frq}(R_K(\varep)_\frn)=0$ (note that we have already proved $F_{PK}(R_P(\varep))_\frn=R_K(\varep)_\frn$, so (PK2), that is, $u_{\frq}(F_{PK}(R_P(\varep))_\frn)=0$ is equivalent to that). We have 
\begin{eqnarray}
u_{\frq}(R_K(\varep)_\frn)&=&(u_{\frq}\wedge\cdots\wedge\varphi_\frq^\frq\wedge\cdots)(\varep_{\frn_i}) \nonumber \\
&=& (u_{\frq}\wedge\cdots\wedge(-u_\frq\cdot x_\frq)\wedge\cdots)(\varep_{\frn_i}) \nonumber \\
&=&0, \nonumber
\end{eqnarray}
where the second equality holds since $\varphi_\frq^\frq=-u_\frq\cdot x_\frq$ by definition (see Definition \ref{defrec}).
\\
(PK3) For any $\frq \in \frn$, we have 
\begin{eqnarray}
v_\frq(R_P(\varep)_\frn)&=&(v_\frq\wedge\cdots\wedge\varphi_\frq\wedge\cdots)(\varep_{\frn_i}) \nonumber \\
&=& (\varphi_\frq \wedge\cdots\wedge(-v_\frq)\wedge\cdots )(\varep_{\frn_i}) \nonumber \\
&=& \varphi_\frq(R_P(\varep)_{\frn/\frq}), \nonumber
\end{eqnarray}
where the second equality is obtained by reversing $v_\frq$ and $\varphi_\frq$ (note that then the sign is changed), and the last by the 
definition of $R_P(\varep)_{\frn/\frq}$. 
\\
(PK4) For any $\frq \in \frn$, we have 
\begin{eqnarray}
R_P(\varep)_\frn|_{\Sigma\setminus\frq}&=& ((\cdots\wedge\varphi_\frq\wedge\cdots)(\varep_{\frn_i}))|_{\Sigma\setminus\frq} \nonumber \\
&=&((\cdots\wedge(-v_\frq\cdot P_\frq)\wedge\cdots)(\varep_{\frn_i}))|_{\Sigma\setminus\frq} \nonumber \\
&=&R_P(\varep)_{\frn/\frq}|_{\Sigma\setminus\frq}\cdot P_\frq, \nonumber
\end{eqnarray}
where the second equality follows by noting $(\cdot)|_{\Sigma\setminus\frq}\circ\varphi_\frq=-v_\frq\cdot P_\frq$. 
\\
(PK5) Note that we have  
$$\varphi_\frq=\pi_\frn\circ \varphi_\frq+(-v_\frq)\cdot P_\frq|_{\Sigma\setminus\frn}$$
for any $\frq\in\frn$. Substitute this into the definition of $R_P(\varep)_\frn$, and expand it, then we have 
$$R_P(\varep)_\frn=\sum_{\frd \subset \frn}\pi_\frn(R_P(\varep)_\frd)\prod_{\frq \in \frn/\frd}P_\frq|_{\Sigma\setminus\frn}.$$
This is (PK5).

\end{proof}

\section{The proof of the main theorem} \label{secpr}

In this section, we prove Theorem \ref{mainthm} by using the general theory developed in \S\S \ref{secKoly} and \ref{secreg}. 
Recall that the setting of the main theorem is the one as in Example \ref{ex1}, so we assume in this section that 7-tuple $(\cO,\Sigma,H,t,v,u,P)$ to be as in Example \ref{ex1}. 

\begin{proposition} \label{propKS}
${\rm{KS}}_1$ and ${\rm{KS}}(A,\cF,\Sigma)$ in \cite[Definition 3.1.3]{MR1} are naturally isomorphic.

\end{proposition} 

\begin{proof}

We use the following fact: there is a natural isomorphism 
$$<\prod_{\ell | n}x_\ell>_{\bbZ/M\bbZ}\stackrel{\sim}{\longrightarrow}\bigotimes_{\ell|n}G_\ell \otimes \bbZ/M\bbZ \quad ; \quad \prod_{\ell|n} x_\ell \mapsto \bigotimes_{\ell|n} \sigma_\ell\otimes 1.$$
For the proof, see \cite[Proposition 4.2 (iv)]{MR2}.

Suppose $\kappa = \{\kappa_n \}_n \in {\rm{KS}}_1$. By (K4) and Corollary \ref{cor1}, we have 
$$\kappa_n \in H\otimes <\prod_{\ell | n}x_\ell>_{\bbZ/M\bbZ},$$
so from the above fact we can naturally regard 
$$\kappa_n \in H\otimes \left( \bigotimes_{\ell|n}G_\ell \right).$$
Since each $G_\ell \otimes \bbZ/M\bbZ$ is isomorphic to $\bbZ/M\bbZ$, we see that $H\otimes \left( \bigotimes_{\ell|n}G_\ell \right)$ is isomorphic to $H$. By this observation, we see that axioms (K1) and (K2) says 
$$\kappa_n \in H^1_{\cF(n)}(\bbQ,A)\otimes \left( \bigotimes_{\ell|n}G_\ell \right).$$
One sees by definition that (K3) is equivalent to the relation in \cite[(5) in Definition 3.1.3]{MR1}. 
Hence we naturally get a Kolyvagin system of \cite{MR1} from our Kolyvagin system. Conversely, the Kolyvagin systems of \cite{MR1} satisfies 
the axioms of our Kolyvagin systems (K1)-(K4), with the identification $<\prod_{\ell | n}x_\ell>_{\bbZ/M\bbZ}=\bigotimes_{\ell|n}G_\ell \otimes \bbZ/M\bbZ$. 
\end{proof}

\begin{theorem} \label{thmsurj}
Suppose the assumptions in Theorem \ref{mainthm} hold. Then the map 
$$R_K:{\rm{US}}_1 \longrightarrow {\rm{KS}}_1$$
is surjective. 
\end{theorem}

\begin{proof}
First note that by Proposition \ref{propKS} we can identify ${\rm{KS}}_1$ and ${\rm{KS}}(A,\cF,\Sigma)$. 
By the proof of \cite[Theorem B.7]{MR1}, there is a $(s,\cU) \in \sI$ for each $m \in\cN$ so that the composed map 
$$\mathop{\ilim}_{n \in \cU}\bigwedge^{\nu(n)+1}S^n \stackrel{R_K}{\longrightarrow} \Im R_K \stackrel{\kappa \mapsto \kappa_m}{\longrightarrow} \cH'(m)$$
is surjective, where $\cH'=\cH'_{(A,\cF,\Sigma)}$ is the sheaf of stub Selmer modules (for the definition, see \cite[Definition 4.3.1]{MR1}).
By the proof of \cite[Corollary 4.3.5]{MR1}, if $m$ is core (see \cite[Definition 4.1.8]{MR1} for definition), then we have an isomorphism 
$$\Gamma(\cH') \stackrel{\sim}{\longrightarrow} \cH'(m) \quad;\quad \kappa \mapsto \kappa_m,$$
where $\Gamma(\cH')$ is the global section of $\cH'$ (see \cite[Definition 3.1.1]{MR1}).
By \cite[Theorem 4.4.1]{MR1}, the natural inclusion $\Gamma(\cH') \hookrightarrow {\rm{KS}}_1$ induces an isomorophism
$$\Gamma(\cH') \stackrel{\sim}{\longrightarrow} {\rm{KS}}_1.$$
Hence we have $\Im R_K ={\rm{KS}}_1$.
\end{proof}

\begin{remark} \label{remcore}
The proof of \cite[Theorem B.7]{MR1} actually shows that we can take $(s,\cU)$ satisfying above so that every $n \in \cU$ is core. We will use this fact later.
\end{remark}

\begin{proposition} \label{propreg}
Suppose $(s,\cU) \in \sI$, $\varep \in \mathop{\ilim}_{n \in \cU}\bigwedge^{\nu(n)+1}S^n$ (see Definition {\rm{\ref{defu}}}), and 
every $n\in\cU$ is core. Then we have for any $n \in \cN$
$$R_T(\varep)_n\in h_n\cR_n.$$
\end{proposition}

This proposition is reduced to the following lemma (note that if $m$ is core, then $h_m=1$):

\begin{lemma}
Suppose $n=\ell_1\cdots \ell_{\nu(n)}$, $m=n\ell_{\nu(n)+1}\cdots\ell_{\nu(m)} \in \cN$. \\If $\varep \in \bigwedge^{\nu(m)+1}H^1_{\cF^m}(\bbQ,A)$, then we have 
$$(\varphi_{\ell_1}^n\wedge\cdots\wedge\varphi_{\ell_{\nu(n)}}^n\wedge(-v_{\ell_{\nu(n)+1}})\wedge\cdots\wedge(-v_{\ell_{\nu(m)}}))(\varep) \in \frac{h_n}{h_m}\cR_n.$$ 
\end{lemma}

\begin{proof}
We prove by induction on $\nu(m/n)$. When $\nu(m/n)=0$, i.e. $m=n$, it is clear by the definition of $\cR_n$ (see Definition \ref{defreg1}). When $\nu(m/n)>0$, put $\ell=\ell_{\nu(m)}$ for simplicity. 

We claim that there are $\varep' \in \bigwedge^{\nu(m/\ell)+1}H^1_{\cF^{m/\ell}}(\bbQ,A)$, $\varep'' \in \bigwedge^{\nu(m)+1}H^1_{\cF^{m/\ell}}(\bbQ,A)$, and $\delta \in H^1_{\cF^m}(\bbQ,A)$ satisfying 
$$\varep=\varep'\wedge\delta + \varep'', \mbox{ and } (v_\ell(\delta))=(\frac{h_{m/\ell}}{h_m}) \ (\mbox{as ideal of }\bbZ/M\bbZ).$$
This claim is shown as follows. First note that by definition we have an exact sequence
$$0 \longrightarrow H^1_{\cF^{m/\ell}}(\bbQ,A) \longrightarrow H^1_{\cF^m}(\bbQ,A) \stackrel{v_\ell}{\longrightarrow} \bbZ/M\bbZ.$$
So we see that there is a $\delta \in H^1_{\cF^m}(\bbQ,A)$ such that $\bar \delta$ generates $H^1_{\cF^m}(\bbQ,A)/H^1_{\cF^{m/\ell}}(\bbQ,A)$. 
Since $v_\ell(\delta)$ generates $\Im(H^1_{\cF^m}(\bbQ,A) \stackrel{v_\ell}{\longrightarrow} \bbZ/M\bbZ) $, we have by the global duality (see \cite[Theorem 2.3.4]{MR1} or \cite[Theorem 1.7.3]{R}) 
$$(v_\ell(\delta))=(\frac{h_{m/\ell}}{h_m}).$$
Since $\bar \delta$ generates $H^1_{\cF^m}(\bbQ,A)/H^1_{\cF^{m/\ell}}(\bbQ,A)$, any $\eta \in H^1_{\cF^m}(\bbQ,A)$ can be written 
as the following form:
$\eta=\eta'+a\delta$, where $\eta'\in H^1_{\cF^{m/\ell}}(\bbQ,A)$ and $a\in\bbZ$. Hence $\varep \in \bigwedge^{\nu(m)+1}H^1_{\cF^m}(\bbQ,A)$ can be written as claimed above. 

By the claim, we have 
\begin{eqnarray}
&& (\varphi_{\ell_1}^n\wedge\cdots\wedge\varphi_{\ell_{\nu(n)}}^n\wedge(-v_{\ell_{\nu(n)+1}})\wedge\cdots\wedge(-v_{\ell_{\nu(m)}}))(\varep) \nonumber \\
&=& \pm v_\ell(\delta) (\varphi_{\ell_1}^n\wedge\cdots\wedge\varphi_{\ell_{\nu(n)}}^n\wedge(-v_{\ell_{\nu(n)+1}})\wedge\cdots\wedge(-v_{\ell_{\nu(m/\ell)}}))(\varep') \nonumber \\
&\in & v_\ell(\delta)\cdot \frac{h_n}{h_{m/\ell}}\cR_n =\frac{h_n}{h_m}\cR_n, \nonumber
\end{eqnarray}
where the first equality follows from that $v_\ell(\varep)=\pm v_\ell(\delta)\varep'$ (by definition), and the next from the inductive hypothesis. Hence we have completed the proof.
\end{proof}



\begin{proposition} \label{proptheta}
$$\{\theta_n(c)\}_n \in {\rm{TKS}}_1.$$

\end{proposition}

\begin{proof}
By (\ref{eq10}) in the proof of Proposition \ref{proptheta1}, we have 
$$\sum_{d|n}(-1)^{\nu(n/d)}\theta_d(c)\prod_{\ell|n/d}P_\ell(\Fr_\ell)=\kappa'_n\otimes\prod_{\ell|n}(\sigma_\ell-1).$$
Note that the left hand side is equal to $F_{TD}(\{\theta_n(c)\}_n)_n$ (see definition \ref{defmor}).
By Theorem \ref{thmKS} and Proposition \ref{propinj}, it is reduced to show 
$$\{ \kappa'_n\otimes\prod_{\ell|n}(\sigma_\ell-1) \}_n \in {\rm{DKS}}_1.$$
(DK1) and (DK3) are well-known properties of Kolyvagin's derivatives (see \cite[Theorem 4.5.1 and Theorem 4.5.4]{R}). (DK2) is shown in \cite[Proof of Theorem 3.2.4 in Appendix A]{MR1} 
(note that 
\begin{eqnarray}
\sum_{d|n}\left( \kappa'_d \otimes \prod_{\ell |d}(\sigma_\ell -1) \right)\cD_{n/d}
&=&\sum_{\tau \in \mathfrak S (n)}\sgn(\tau) \left(\kappa'_{d_\tau} \otimes \prod_{\ell | d_\tau}(\sigma_\ell -1) \right)\prod_{\ell|n/{d_\tau}}\pi_\ell(P_{\tau(\ell)}(\Fr_{\tau(\ell)}^{-1})), \nonumber
\end{eqnarray}
where $\mathfrak S(n)$ is the set of permutations of the prime divisors of $n$, and $d_\tau=\prod_{\tau(\ell)=\ell}\ell$). (DK4) is clearly satisfied.
\end{proof}

\begin{remark} \label{remkappa}
From the above, we see that the Kolyvagin's derivative class $\kappa_n'$ satisfies 
$$\kappa_n' \otimes \prod_{\ell|n}(\sigma_\ell -1) = s_n(\theta_n(c))$$
(for the definition of $s_n$, see Definition \ref{defs}). So if we admit Theorem \ref{mainthm}, then we have 
$$s_n(\theta_n(c)) \in h_n s_n(\cR_n) \subset h_n H^1_{\cF^n}(\bbQ, A) \otimes < \prod_{\ell|n}(\sigma_\ell-1)>.$$
Hence we have the following upper bound of $h_n$:
$$\ord_p(h_n) \leq \sup \{ m \ | \ \kappa_n' \in p^m H^1_{\cF^n}(\bbQ,A) \}.$$
This generalizes Corollary \ref{corrubin}, since $\kappa_1'=c_\bbQ$.
\end{remark}

Now we prove the main theorem.

\begin{proof}[Proof of Theorem \ref{mainthm}]
By Proposition \ref{proptheta}, Theorem \ref{thmKS}, Theorem \ref{thmreg} and Theorem \ref{thmsurj}, there exists $\varep \in\mathop{\ilim}_{n \in \cU}\bigwedge^{\nu(n)+1}S^n$ such that 
$$R_T(\varep)_n=\theta_n(c).$$
Here note that by Remark \ref{remcore} every $n\in\cU$ is taken to be core. Hence by Proposition \ref{propreg} we have 
$$R_T(\varep)_n \in h_n \cR_n.$$
This completes the proof.
\end{proof}

\section*{Acknowledgement}
The author would like to thank Professor Masato Kurihara for giving him many valuable suggestions and helpful advice.

\end{document}